\documentclass[11pt,a4paper,twoside,reqno]{amsart}

\usepackage{amssymb}
\usepackage{latexsym}

\usepackage{graphicx}
\usepackage{a4wide}

\usepackage{amsmath}
\usepackage{amsthm}

\usepackage{color}
\usepackage{caption}
\usepackage{subcaption}
\usepackage{hyperref}
\usepackage{multirow}
\usepackage{booktabs}

\newtheorem{theorem}{Theorem}
\newtheorem{lemma}{Lemma}

\newcommand{\norm}[1]{\left\| #1 \right\|}

\newcommand{\eps}{\varepsilon}

\newcommand{\pt}{\partial_t}
\newcommand{\ptt}{\partial^2_{tt}}

\renewcommand{\i}{\ifmmode\mathit{\mathchar"7010 }\else\char"10 \fi}
\renewcommand{\j}{\ifmmode\mathit{\mathchar"7011 }\else\char"11 \fi}

\newcommand{\Z}{\mathbb{Z}}

\newcommand{\vfi}{\varphi}

\begin{document}\large

\title{A space-time discretization of a nonlinear peridynamic model on a 2D lamina}%

\author[L. Lopez]{Luciano Lopez}

\author[S. F. Pellegrino]{Sabrina Francesca Pellegrino
}

\address[L. Lopez]{Dipartimento di Matematica, Universit\`a degli Studi di Bari Aldo Moro, via E. Orabona 4, 70125 Bari, Italy}
\email{luciano.lopez@uniba.it}

\address[S. F. Pellegrino]{Dipartimento di Management, Finanza e Tecnologia, Universit\`a LUM Giuseppe Degennaro, S.S. 100 Km 18 - 70010 Casamassima (BA), Italy}
\email{pellegrino@lum.it}

\begin{abstract}
Peridynamics is a nonlocal theory for dynamic fracture analysis consisting in a second order in time partial integro-differential equation. In this paper, we consider a nonlinear model of peridynamics in a two-dimensional spatial domain. We implement a spectral method for the space discretization based on the Fourier expansion of the solution  while we consider  the Newmark-$\beta$ method for the time marching.  This computational approach takes advantages from the convolutional form of the peridynamic operator and from the use of the discrete Fourier transform. We show a convergence result for the fully discrete approximation  and study the stability of the method applied to the linear peridynamic model. Finally, we perform several numerical tests and comparisons to validate our results and provide simulations implementing a volume penalization technique to avoid the limitation of periodic boundary conditions due to the spectral approach. 
\end{abstract}

\maketitle





\section{Introduction}
\label{sec:statement}

Complex fracture problems require accurate prediction of damage behavior or spontaneous cracks formation in anisotropic materials. Classical theory of continuum mechanics is unsuitable for modeling discontinuous phenomena, such as cracks and defects, because it requires the partial derivatives of the displacement field to be known all over the domain, but they do not exist on discontinuities. Non-local theories allow a unique equation to be used both on or off a crack, see~\cite{EE,DSDPR,CoclitePellegrino,BerardiAWR,Pellegrino,DELIA2021}, and recent studies show that differential operators of fractional orders may depict the nature of such phenomena (see for instance~\cite{CDMV,GarrappaSIAM,Garrappa2018,Garrappa2015115,DELIA2013}). 

Peridynamics is a non-local version of the elasticity theory introduced by Silling in~\cite{Silling_2000} to solve discontinuous problems without using partial derivatives. In the bond-based formulation, the motion of a material body is governed by an integro-differential partial equation, where each infinitesimal unit of continuum interacts with other units in its neighborhood directly across finite distance. The use of integral-differential equations instead of the spatial differential equations allow the displacement and internal forces to develop singularities (see~\cite{L,QYX,LSJ,BAC,Er,Alebrahim2016,WECKNER2005705,Delia2017}). 

The theory is non-local because the interactions between material points extend beyond their neighborhood inside a region with finite radius called horizon (see for instance~\cite{Bobaru_2009}). This feature makes it possible to analyze fracture problems involving viscoelastic and cohesive materials.

We fix $[0,T]$, for some $T>0$, as the time domain under investigation. Consider a continuum body with mass density $\rho: V\times[0,T] \to {\mathbb{R}}_+$ occupying a region $V\subset {\mathbb{R}}^2$. Then, the peridynamic model describes the dynamics of the body and its equation is given by
\begin{equation}
\label{eq:perid}
\rho( x) \partial_{tt}^2{ u}( x,t) = \int_{V}  f( x'- x,  u( x',t) -  u( x,t))d x' +  b( x,t),\quad  x\in V,\quad t\in[0,T],
\end{equation}
with initial conditions
\begin{equation}
\label{eq:initcond}
 u( x,0)= u_0(x),\quad \partial_t  u( x,0)= v_0( x),\qquad  x\in V,
\end{equation}
where $ u$ is the displacement field and $ b$ describes all the external forces acting on the material body. The interaction between two material points is described by a response function $f$, called {\bf pairwise force function}, that contains the constitutive law associated with the material. This means that the integrand $ f$ denotes the force density that the particle $ x'$ exerts on the particle $ x$, see for instance~\cite{Silling_2000}. The interaction between $ x$ with all particle in its peridynamic neighborhood is called {\bf bond}, see Figure~\ref{fig:domain-horizon}. We set
\[ \xi =  x' - x,\qquad \eta =  u( x',t) -  u( x,t),\]
which denote the relative position of two particles in the reference configuration and the relative displacement, respectively. Thus $ \xi + \eta$ represents the current relative position vector, and we notice that the pairwise force function $ f$ satisfies Newton's third law and the conservation of the angular momentum:
\begin{equation}
\label{eq:property}
 f(-\xi,-\eta) = - f(\xi,\eta),\qquad \eta\times  f(\xi,\eta) =0.
\end{equation}

\begin{figure}
\centering
\includegraphics[width=.4\textwidth]{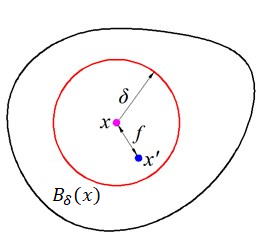}
\caption{The peridynamic domain and the horizon.}
\label{fig:domain-horizon}
\end{figure}

Since peridynamics prescribes finite-range interactions, we assume the existence of a positive cutoff constant $\delta$, such that there are no interactions among material points having relative distance greater than $\delta$ (see~\cite{WECKNER2005705}). Indeed, the state of a material point is influenced by all points in a region of finite radius called {\bf horizon} namely
\[ f(\xi,\eta) = 0,\qquad \text{for $\,|\xi| >\delta\,$ and for every $\eta$.}\]

The value of $\delta$ is a parameter that represents the locality of the interactions. The interactions become more local as $\delta$ decreases. Thus, the classical theory can be thought as the limiting case of peridynamics as $\delta$ goes to zero. 

Additionally, the non-linear peridynamic operator in~\eqref{eq:perid} can be understood as
\[\int_{V}  f( x'- x,  u( x',t) -  u( x,t))d x'  = \int_{V\cap B_{\delta}(x)}  f( x'- x,  u( x',t) -  u( x,t))d x',\]
where $B_{\delta}(x) = \{x'\in\mathbb{R}\,:\,|x-x'|\le \delta\}$.

In what follows, we restrict our attention to the case of an homogeneous bi-dimensional lamina, where the evolution of the material body is given by a class of {\bf nonlinear} peridynamic pairwise force function of convolution type in separable form 
\begin{equation}
\label{eq:f}
 f(\xi,\eta) =  C (\xi) w (\eta),
\end{equation}
where the function $C$ is a non-negative even function, i.e. $C(-\xi) =C(\xi)$, called {\bf micromodulus function}. We assume that $ C ( x,  x') \equiv 0$ for $| x -  x'|>\delta$. While  $ w$ is an odd global Lipschitz continuous function for which there exists a non-negative function $\ell \in L^1(B_{\delta}(0)) \cap L^{\infty} (B_{\delta}(0))$ such that for all $ \xi\in\mathbb{R}^2$, with $| \xi|\le\delta$ and $ \eta$, $\eta'$ there holds
\[| w( \eta') -  w( \eta)| \le\ell(\xi)|\eta'-\eta|.\]

Thus, the model becomes
\begin{equation}
\label{eq:nonlinperi1}
\rho(x) \partial_{tt}^2 u(x,t) = \int_{B_{\delta}(x)} C(x' - x) w\left(u(x',t)- u(x,t)\right)d x' + b(x,t),
\end{equation}
for $ x\in V,\ t\in[0,T],$ with initial conditions
\begin{equation}
\label{eq:incon1}
u(x,0)=u_0(x),\quad \pt u(x,0)=v_0(x),\qquad x\in V.
\end{equation}

In particular, we focus on the case
\begin{equation}
\label{eq:w}
w(\eta) = \eta^r,\quad r\text{ odd},\quad r\ge1.
\end{equation}

We observe that when $r=1$, we obtain the linear case studied in~\cite{CFLMP,WECKNER2005705,Pellegrino2020}. Instead values of $r$ greater than one are useful both from an analytical and a physical point of view, as the power-type nonlinearity in the pairwise force function resembles a fractional derivative (see for instance~\cite{CDMV,LP}), and the well-posedness of the model is achieved in this setting (see~\cite{erbay2012,CDMV}). Additionally, it could be easily generalized to the following more common nonlinearities used in~\cite{CDMV}:
\[f(\xi,\eta) = \frac{|\eta|^{p-2} \eta}{|\xi|^{2+\alpha p}},\qquad p\ge2,\quad\alpha\in(0,1).\]

If we define the nonlinear peridynamic operator of~\eqref{eq:nonlinperi1} as follows
\begin{equation}
\label{eq:L}
\mathcal{L}(u(x,t)) = \int_{V} C( x' - x)\left(u(x', t) - u(x, t)\right)^r\, d x', \quad x\in V,\ t\in[0,T],
\end{equation}
then equations~\eqref{eq:nonlinperi1} and~\eqref{eq:incon1} become
\begin{equation}
\label{eq:nonlinperid}
\begin{cases}
\rho(x) \partial_{tt}^2 u(x,t) = \mathcal{L}(u(x,t)) + b(x,t),\qquad &x\in V,\ t\in[0,T],\\
u(x,0)=u_0(x),\quad v(x,0)=v_0(x),\qquad &x\in V,
\end{cases}
\end{equation}
where $v(x,t) = \pt u(x,t)$.

In order to solve complex problems using the peridynamic theory, a numerical approach is necessary. To discretize in space the peridynamic equation, the most implemented methods are the finite element methods and meshfree methods (see for instance~\cite{Silling_Askari_2005,CFLMP}). Instead, spectral methods, based on truncated Fouries series in space,  result to be very accurate and suitable in the presence of nonlocalities. These techniques rewrite the equations in the Fourier space, transforming derivatives and convolution products into multiplication and reducing the total computational cost of the procedure by using the discrete Fourier transform (DFT)  and the Fast Fourier transform (FFT)  algorithm (see for instance~\cite{LP,CFLMP,Jafarzadeh}).

On the other hand, the time integration of the model can be done by using explicit forward and backward difference techniques (see~\cite{Macek,Lapidus_2003,KILIC2010}). The St\"ormer-Verlet method consists in an explicit central second-order finite difference scheme widely used in elastodynamics and in the context of wave propagation (see for example~\cite{Galvanetto2018,LP,CFLMP,hairer2003,Galvanetto2016}). It is a robust and symplectic scheme simple to implement which preserves geometric properties of the flow, such as the energy of the system, but it requires a restriction on the step size. 

The numerical study of non-local models demands for high accurate solutions, and as explained before, spectral collocation methods can achieve good accuracy. However, the application of explicit time marching schemes make the implementation of spectral space discretization very expansive when we need to study the long time behavior. The implicit time schemes can provide the same accuracy of the explicit ones, but using greater time steps. The Newmark-$\beta$ method, for $0<\beta\le 1/2$, is an implicit second order integrator largely used in continuum mechanics and for structural dynamic problems. It depends on a parameter $\beta$ which let the acceleration of the system to vary in the time interval under consideration. It is unconditionally stable in time for $\beta\in[1/4,1/2]$, and has computational advantages compared to the explicit methods, particularly as problems become stiff (see~\cite{ZAMPIERI2006,LAIER2011}).

In this paper, we apply spectral methods based on the Fourier expansion for the spatial discretization of the 2D peridynamic model~\eqref{eq:nonlinperid} and perform the time integration by the Newmark-$\beta$ method instead of the more standard St\"ormer-Verlet method.

The paper is organized as follows. In Section~\ref{sec:spectral} we describe the spectral Fourier collocation method for spatial discretization of the bi-dimensional domain and we observe that the computational cost can be reduced as the method moves the convolution product to a multiplication. We also recall a convergence result for the semi-discrete problem. Section~\ref{sec:timediscret} introduce the Newmark-$\beta$ method for time marching and contain a proof for the convergence of the fully discrete problem for the case $r=1$ and a stability analysis of the method. Section~\ref{sec:test} is devoted to numerical simulations both with and without the implementation of a volume penalization. A validation is computed by comparing the exact and the numeric solution and by analyzing the relative $L^2$-error. Moreover, we provide comparisons between the performance of the Newmark-$\beta$ method and the St\"ormer-Verlet method used for the time discretization of the model.

\section{Spectral semi-discretization of the problem}
\label{sec:spectral}

In the framework of engineering computation, spectral methods represent a good strategy for the global discretization of partial differential equations as they guarantee high levels of accuracy even when applied to nonlinear problems or when long time integration is necessary (see for instance~\cite{Canuto2006}). 

To obtain a spectral discretization of the spatial domain one can consider a Fourier series expansion of the solution $u(x, t)$ and then makes a truncation of the obtained series expansion. The method requires the assumption of periodic boundary conditions and is often implemented in peridynamic problems where a convolution product appears in the nonlinear integral operator $\mathcal{L}$. Indeed spectral methods allow to transform convolutions to multiplications (see~\cite{CFLMP,LP,alebrahim,Emmrich_Weckner_2007_2,Jafarzadeh,ZHAO2019,Pellegrino2020,Bobaru2021}). For problems with non periodic boundary conditions, one can employ volume penalization techniques as proposed in~\cite{LP,Jafarzadeh}. While, extensions of this spectral discretization to irregular domain are possible (see for instance~\cite{Orovio2006,Guimaraes2020}).

The discretization of 1D spatial domain by means of spectral Fourier methods in the context of peridynamic models have been performed for example in~\cite{LP,CFLMP}.

Instead, in this paper, we consider the spatial domain which is a 2D lamina of $\mathbb{R}^2$ given by $V=[a,b]\times[a,b]$. 

We assume that the mass density is constant in space, and to simplify the notation, we suppose $\rho(x)\equiv 1$. Let $w(\eta)=\eta^r$, for $r$ odd and $r> 1$ and $\delta>0$ be the horizon.

Using the following definition of the periodic convolution product, 
\begin{equation*}
\label{eq:periodic_convolution}
 C\ast_V  u = \int_{V}  C( x- x') u( x',t)\,d x',
\end{equation*}
we rewrite the model~\eqref{eq:L}-\eqref{eq:nonlinperid} as
\begin{align}
\label{eq:peri-convolution}
\ptt  u =& ( C\ast_V  u^r) + \sum_{\ell=1}^{r-1}\binom{r}{\ell}(-1)^{\ell} u^{\ell}\left( C\ast_V  u^{r-\ell}\right)- \gamma u^r +  b\  ,
\end{align}
for  $  x\in V$, $t\in[0,T]$ and where  $\gamma = \int_{-\infty}^{+\infty} \int_{-\infty}^{+\infty} C( x)\,d x$. 
Indeed,
\begin{align*}
\ptt  u(x,t) =& \int_V  C( x- x')\left( u( x',t)- u( x,t)\right)^r\,d x' +  b( x,t)\\
=& \sum_{\ell=0}^r \binom{r}{\ell} (-1)^{\ell}  u^\ell( x,t) \int_V  C( x- x') u^{r-\ell}( x',t)\,d x' +  b( x,t)\\
=& ( C\ast_V  u^r)( x,t) + \sum_{\ell=1}^{r-1}\binom{r}{\ell}(-1)^{\ell} u^{\ell}( x,t)\left( C\ast_V  u^{r-\ell}\right)( x,t)\\
& - \gamma  u^r( x,t) +  b( x,t).
\end{align*}

Hence, the nonlinear peridynamic operator~\eqref{eq:L} becomes
\begin{equation}
\label{eq:Lnew}
\mathcal{L}(u) = ( C\ast_V  u^r) + \sum_{\ell=1}^{r-1}\binom{r}{\ell}(-1)^{\ell} u^{\ell}\left( C\ast_V  u^{r-\ell}\right) -\gamma  u^r, 
\end{equation}
for $x\in V,\ 0\le t\le T$. 

Let $ u( x_1,x_2,t)$ be a real-valued function defined over the periodic domain $V$. Then we can express $ u( x_1,x_2,t)$ by the infinite Fourier series in space
\begin{equation}
\label{eq:fourierseries2d}
  u(x_1,x_2,t) = \sum_{k_1=-\infty}^\infty \sum_{k_2=-\infty}^\infty \hat{ u}(k_1,k_2,t) e^{\Im (k_1 x_1 + k_2 x_2)}, 
\end{equation}
where $(x_1,x_2)\in V$, $t\in[0,T]$, [$\Im$ denotes the imaginary unit $\Im = \sqrt{-1}$].  In \eqref{eq:fourierseries2d} $\hat{ u}(k_1,k_2,t)$ for $k=(k_1,k_2)$ with $k_1$, $k_2\in\Z$  and  $t\in[0,T]$ represents the 2D Fourier coefficients of $u$:
\begin{equation}
\label{eq:coeff2d}
\hat{ u}(k_1,k_2,t) = \int_a^b\int_a^b  u(x_{1},x_{2},t) e^{-\Im (k_1x_{1}+k_2 x_{2})}dx_1dx_2.
\end{equation}
The form~\eqref{eq:fourierseries2d} is the 2D inverse Fourier transform $\mathcal{F}^{-1}$, while equation~\eqref{eq:coeff2d} represents the Fourier transform $\mathcal{F}$ of $u$.

Thanks to the Convolution Theorem, we can compute the periodic convolution in~\eqref{eq:peri-convolution} by means of the inverse Fourier transform $\mathcal{F}^{-1}$ of the product of Fourier coefficients:
\begin{equation}
\label{eq:convth}
 C\ast_V\  u^{r} = \mathcal{F}^{-1}\left(\mathcal{F}(C)\mathcal{F}\left(u^{r}\right)\right).
\end{equation}
Additionally, according to the Inverse Theorem, we obtain
\begin{equation}
\label{eq:Invth}
u^{\ell}\left( C\ast_V u^{r-\ell}\right) = \mathcal{F}^{-1}\left(\mathcal{F}\left( u^{\ell}\right)\ast_V\left(\mathcal{F}( C) \mathcal{F}\left( u^{r-\ell}\right) \right)\right).
\end{equation}

Thus, using~\eqref{eq:convth} and~\eqref{eq:Invth}, the equation~\eqref{eq:peri-convolution} becomes
\begin{align}
\label{eq:method}
\ptt  u =& \mathcal{F}^{-1}\left(\mathcal{F}(C)\mathcal{F}\left( u^r\right) \right) + \sum_{\ell=1}^{r-1} \binom{r}{\ell} (-1)^{\ell} \mathcal{F}^{-1}\left(\mathcal{F}\left( u^{\ell}\right)\ast_V\left(\mathcal{F}( C) \mathcal{F}\left( u^{r-\ell}\right) \right)\right)\\
& -\gamma  u^r +  b.\notag
\end{align}

As a consequence, the integral peridynamic operator $\mathcal{L}$ in~\eqref{eq:Lnew} can be rewritten as follows:
\begin{align}
\label{eq:Lnew1}
\mathcal{L}( u) =& \mathcal{F}^{-1}\left(\mathcal{F}(C)\mathcal{F}\left( u^r\right) \right)  + \sum_{\ell=1}^{r-1} \binom{r}{\ell} (-1)^{\ell} \mathcal{F}^{-1}\left(\mathcal{F}\left( u^{\ell}\right)\ast_V\left(\mathcal{F}( C) \mathcal{F}\left( u^{r-\ell}\right) \right)\right)\\
& - \gamma u^r.\notag
\end{align}

In order to construct the spectral method for~\eqref{eq:method}, we have to approximate $u$ at the collocation points by its truncated Fourier series. Let $\Delta x>0$ be the space step in both directions. We discretize the spatial domain  $V=[a,b]\times[a,b]$ by the equidistant collocation points $ x_{ n} = \left(x_{n_1},\ x_{n_2}\right)\in V$, with $ n =\left(n_1,\ n_2\right)$, such that
\[
x_{n_1} = a + n_1\Delta x,\quad x_{n_2}=a+n_2\Delta x,\quad \text{for}\quad n_1,\,n_2\in\{0,\dots,N\},
\]
where $N=\left\lfloor \frac{b-a}{\Delta x}\right\rfloor$, see Figure~\ref{fig:discret}.

\begin{figure}
\centering
\includegraphics[width=.4\textwidth]{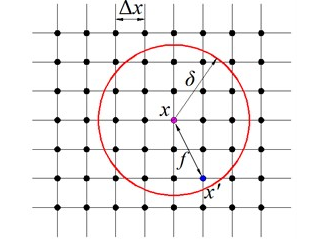}
\caption{The spatial discretization of the peridynamic domain.}
\label{fig:discret}
\end{figure}

Then we can approximate $u$ by the truncated Fourier series $ u^N$
\begin{equation}
\label{eq:fourierseries2d1}
  u^N(x_{1},x_{2},t) = \sum_{k_1=-N}^N \sum_{k_2=-N}^N \tilde{ u}(k_1,k_2,t) e^{\Im (k_1 x_{1} + k_2 x_{2})}, 
\end{equation}
for $t\in[0,T]$. In \eqref{eq:fourierseries2d1} $\tilde{ u}(k_1,k_2,t)$ for $k=(k_1,k_2)$ with $k_1$, $k_2\in\{-N,\dots, N\}$  and  $t\in[0,T]$ represents the 2D discrete Fourier transform (DFT)
\begin{equation}
\label{eq:coeff2d1}
\tilde{ u}(k_1,k_2,t) = \frac{1}{(N+1)^2 c_{k_1} c_{k_2}} \sum_{n_1=0}^{N} \sum_{n_2 =0}^{N}  u^N(x_{n_1},x_{n_2},t) e^{-\Im (k_1x_{n_1}+k_2 x_{n_2})},
\end{equation}
where
\[
c_{k_i} = \begin{cases}
2,&\text{if $k_i=\pm N$},\\
1,&\text{otherwise},
\end{cases}
\quad i=1,\,2.
\]
The form \eqref{eq:fourierseries2d1} evaluated in $(x_{n_1}, x_{n_2})$ is the 2D  inverse discrete Fourier transform (IDFT). 

We notice that the truncated Fourier series $u^N(x_1,x_2,t)$ converges to $u(x_1,x_2,t)$ as $N$ goes to infinity. Moreover, we have that $u^N(x_1,x_2,t)$ represents a discrete interpolant of $u$, in fact 
\[u^N(x_n, t) = u(x_n,t),\]
for $n=(n_1,n_2)$, with $n_1$, $n_2\in\{ 0,\cdots, N\}$, and $t\in[0,T]$   (see~\cite{Canuto2006}).

Often, for the sake of simplicity,  we will use  the following notation: $u^N( x,t)$ instead of  $ u^N(x_1,x_2,t)$  with $x=(x_1,x_2)\in V$  and  $\tilde{u}_k(t)$ instead of $\tilde{ u}(k_1,k_2,t)$ for every $k=(k_1,k_2)$  with $k_1$, $k_2\in\{-N,\dots, N\}$. Moreover, to lighten the notation, we denote the 2D discrete Fourier transform by $\mathcal{F}_N$ and the 2D inverse discrete Fourier transform by $\mathcal{F}_N^{-1}$.


By using the Fourier collocation method  and the definition of the truncated Fourier series, if we replace $u(x,t)$ in ~\eqref{eq:method} with  $u^N(x,t)$  in ~\eqref{eq:fourierseries2d1} and evaluate $u^N(x,t)$ at $x_n$, we obtain the discrete form of the peridynamic operator $\mathcal{L}$ in~\eqref{eq:Lnew1}:

\begin{align}
\label{eq:L-discret}
\mathcal{L}_N( u_n^N) =& \left(\mathcal{F}_N^{-1}\left(\mathcal{F}_N(C)\mathcal{F}_N\left( \left(u_n^N\right)^r\right) \left(\Delta x\right)^2\right)\right) \\ 
& + \left(\sum_{\ell=1}^{r-1} \binom{r}{\ell} (-1)^{\ell} \mathcal{F}_N^{-1}\left(\mathcal{F}_N\left( \left(u_n^N\right)^{\ell}\right)\ast_V\left(\mathcal{F}_N( C) \mathcal{F}_N\left( \left(u_n^N\right)^{r-\ell}\right) \left(\Delta x\right)^2\right)\right)\right)\notag \\
& - \gamma  ( u_n^N)^r\ ,\notag
\end{align}
where $u^N_n(\cdot)$ approximates $u^N(x_n,\cdot)$.

Thus, the spectral semi-discrete method for~\eqref{eq:nonlinperid} becomes
\begin{equation}
\label{eq:compact1}
\begin{cases}
\frac{d^2}{dt^2}  u_n^N = \mathcal{L}_N ( u_n^N)+  b_n, \qquad  \qquad t\in [0,T]\\
 u_n^N(0) =u_0(x_n),\  v_n^N(0)=v_0(x_n),
\end{cases}
\end{equation}
where
\[
 u_0( x_{ n}) = \sum_{ k_1=-N}^N\sum_{k_2=-N}^N \tilde{ u}_{0, k} e^{\Im \left( k_1 x_{ n_1}+ k_2  x_{ n_2}\right)},\qquad v_0( x_{ n}) = \sum_{ k_1=-N}^N\sum_{k_2=-N}^N \tilde{ v}_{0, k} e^{\Im \left( k_1 x_{ n_1}+ k_2  x_{ n_2}\right)},
\]
for each $n=(n_1,n_2)$ with $n_1,\,n_2\in\{0,\dots, N\}$.

The proposed spectral semi-discretization method~\eqref{eq:compact1} can benefit of the Fast Fourier transform (FFT) in order to reduce efficiently its computational cost. Indeed, if we numerically compute the discrete Fourier transform $\mathcal{F}_N$, which appears in~\eqref{eq:L-discret} by means of the FFT function, we find that the complexity of the method is $\mathcal{O}(N^2\log_2^2(N))$ compared with $\mathcal{O}(N^4)$ for the conventional quadrature formula or peridynamic meshfree and finite element solvers of the 2D problems.


We recall that, due to the interpolant nature of $u^N$, the spectral method~\eqref{eq:compact1} is locally constructed in such a way, on each collocation point, we have
\[u^N_n(t) \approx u(x_n,t),\qquad n=(n_1,n_2),\quad n_1,\,n_2\in\{ 0,\cdots, N\},\]
where $u(\cdot,t)$ is the solution of the problem~\eqref{eq:nonlinperid} at time $t\in[0,T]$.

For the time discretization of this system of ODEs we will consider the Newmark-$\beta$ method and the approximation of $u^N_n(t)$ at a point $t_s$ of the mesh on $[0,T]$ will be denoted by $u^N_{n,s}$.

\subsection{Convergence of the semi-discrete approximation}

In this section we present a convergence result for the spectral semi-discrete problem essentially similar to the one given for the one-dimensional case in~\cite{LP}. In what follows, $M$ denotes a generic positive constant. We denote by $(\cdot,\cdot)$ and $\norm{\cdot}$ the inner product and the norm of $L^2(V)$, respectively, namely, if $u$, $v\in L^2(V)$, then
\[
(u,v) = \int_{V} u(x)v(x)\,dx,\qquad \norm{u}^2 = (u,u).
\]
Let $S_N$ be the space of trigonometric polynomials of degree $N$,
\[
S_N = \text{span}\left\{e^{\Im \left(k_1 x_1 + k_2 x_2\right)}| -N\le k_1,\,k_2\le N, \quad x_1,\,x_2\in[a,b]\right\},
\]
and $P_N: L^2(V) \to S_N$ be an orthogonal projection operator
\[
P_N u(x) = \sum_{k_1=-N}^N\sum_{k_2=-N}^N \tilde{ u}_k e^{\Im \left(k_1 x_1 + k_2 x_2\right)},
\]
such that for any $u\in L^2(V)$, the following equality holds
\begin{equation}
\label{eq:orthogonal}
(u-P_Nu,\varphi) = 0,\quad\text{for every $\varphi\in S_N$}.
\end{equation}
The operator $P_N$ commutes with derivatives in the distributional sense:
\[
\partial_x^q P_N u = P_N\partial_x^q u.
\]
Moreover, for the duality relation between the operators, $P_N$ satisfies
\begin{equation}
\label{eq:duality}
P_N \mathcal{L} = \mathcal{L}_N,\quad\text{and}\quad P_N\mathcal{L}_N = \mathcal{L}.
\end{equation}

We denote by $H^s_p(V)$ the periodic Sobolev space and by $X_s = \mathcal{C}^1\left(H^s_p(V);[0,T]\right)$ the space of all continuous functions in $H_p^s(V)$ whose distributional derivative is also in $H_p^s(V)$, with norm
\[
\norm{u}_{X_s}^2 = \max_{t\in[0,T]}\left(\norm{u(\cdot,t)}^2 + \norm{\partial_t u(\cdot,t)}^2\right), \qquad u\in X_s\ .
\]

The spectral scheme for~\eqref{eq:nonlinperid} with periodic boundary conditions is
\begin{align}
\label{eq:scheme}
\ptt u^N &= P_N \mathcal{L}(u^N) + b
,\\
\label{eq:initial_scheme}
u^N(x,0) &= P_N u_0(x),\quad v^N(x, 0) = P_N v_0(x),
\end{align}
where $u^{N}(\cdot,t)\in S_N$ for every $0\le t\le T$.

The following lemmas are preliminary to the convergence result of the semi-discrete scheme.
\begin{lemma}[see~\cite{Canuto2006}]
\label{lm:rhostima}
For every real $0\le \mu\le s$, there exists a positive constant $L$ such that
\begin{equation}
\label{eq:sobolev}
\norm{u-P_N u}_{H_p^{\mu}(V)} \le L N^{\mu-s}\norm{u}_{H^s_p(V)}, \quad\text{for every $u\in H_p^{s}(V)$}.
\end{equation}
\end{lemma}

\begin{lemma}[see~\cite{Duo2019}]
\label{lm:spectral}
The spectrum of the discrete peridynamic operator $-\mathcal{L}_N$ satisfies the following condition
\begin{equation*}
sp(-\mathcal{L}_N) \subseteq[\lambda_{*},\lambda^{*}\left(\Delta x\right)^2],
\end{equation*}
where $\lambda_{*}$ and $\lambda^{*}$ are positive constants.
\end{lemma}

\begin{theorem}
\label{th:convergence}
Let $u\in X_s$, $s\ge 1$, be the solution of the problem~\eqref{eq:nonlinperid} with periodic boundary conditions and initial conditions $u_0$, $v_0\in H_p^s(V)$. Let $u^N$ be the solution of the semi-discrete scheme~\eqref{eq:scheme}-\eqref{eq:initial_scheme}. Assume that $C\in L^{\infty}(V)$, then, for every $T>0$, there exists a constant $M = M(T)$, independent on $N$, such that
\begin{equation}
\label{eq:order_conv}
\norm{u-u^N}_{X_1} \le M(T) \left(\Delta x\right)^{s-1} \norm{u}_{X_s}.
\end{equation}

\end{theorem}
For the proof of Theorem~\ref{th:convergence}, we can easily extend to the bi-dimensional case the convergence result of~\cite{LP}. 


\section{The fully discrete problem}
\label{sec:timediscret}

Here we derive the fully discretization of the peridynamic equation~\eqref{eq:nonlinperid} by using the Newmark-$\beta$ method, which is an implicit integrator of the second order in time, largely used in various fields of engineering, in particular in dynamic response systems, elastodynamics and in the context of partial differential equation of wave propagation (see~\cite{ZAMPIERI2006,LAIER2011}). It is implicit for $0<\beta\le1/2$, but it offers the advantage to be unconditionally stable in time   when $\beta\in[1/4,1/2]$.

Let $\Delta t >0$ be the time step and we partition the time interval $[0,T]$ by means of the discrete sequence  $t_s = s\Delta t$, for $s=0,\dots,S_T$, where $S_T = \left\lfloor \frac{T}{\Delta t}\right\rfloor$. We denote by $(u_s^N(\cdot),v_s^N(\cdot))$ the numerical approximation of $(u^N(\cdot,t_s),v^N(\cdot,t_s))$ so that   $(u_s^N(x_n),v_s^N(x_n))=(u_{n,s}^N,v_{n,s} ^N)$.  

For the sake of simplicity, we assume $b\equiv 0$ and $\rho\equiv 1$. We apply the Newmark-$\beta$ method to the semi-discrete problem~\eqref{eq:compact1} by using an extended version of the Cauchy's mean value theorem. The displacement first derivative can be approximated as follows:

\begin{equation}
\label{eq:ut}
v_{s+1}^N = v_s^N + \frac{\Delta t}{2}\left(\mathcal{L}_N(u_s^N) + \mathcal{L}_N(u_{s+1}^N)\right),
\end{equation}
while we obtain the following expression for the displacement
\begin{equation}
\label{eq:utt}
u_{s+1}^N = u_s^N + \Delta t v_s^N + \frac{(\Delta t)^2}{2} u_{tt,\beta}^{N},
\end{equation}
where 
\begin{equation}
\label{eq:ubeta}
u_{tt,\beta}^{N} = (1-2\beta) u_{tt,s+1}^{N} + 2\beta u_{tt,s}^N,\qquad 0\le2\beta\le1,
\end{equation}
and $u_{tt,s}^N$ denotes the second derivative in time of $u^N$ evaluated in $t_s$.

The introduction of the parameter $\beta$ allows the acceleration to vary as $\beta$ varies, and as we will see later, there exists an interval of values for $\beta$ that guarantees the convergence of the fully-discrete problem.

Substituting~\eqref{eq:ubeta} into~\eqref{eq:compact1} and collecting equations~\eqref{eq:utt} and~\eqref{eq:ut}, we get the final expression of the method:
\begin{equation}
\label{eq:newmarkb}
\begin{cases}
u_{s+1}^N = u_s^N + \Delta t v_s^N + (\Delta t)^2\left(\left(\frac{1}{2}-\beta\right)\mathcal{L}_N (u_s^N) + \beta \mathcal{L}_N (u_{s+1}^N)\right),\\
v_{s+1}^N = v_s^N + \frac{\Delta t}{2}\left(\mathcal{L}_N (u_s^N) + \mathcal{L}_N(u_{s+1}^N)\right),\\
u_0^N = u_{n,0},\quad v_0^N = v_{n,0}.
\end{cases}
\end{equation}

Additionally, we can express system~\eqref{eq:newmarkb} in the following way, by eliminating $v_s^N$ and $v_{s+1}^N$:
\begin{align}
\label{eq:trinomial}
\frac{u_{s+1}^N - 2 u_s^N + u_{s-1}^N}{(\Delta t)^2} &=\frac{1}{\Delta t} \left(\frac{u_{s+1}^N-u_s^N}{\Delta t} - \frac{u_s^N - u_{s-1}^N}{\Delta t}\right)\\
&= \frac{1}{\Delta t}\left(v_s^N - v_{s-1}^N\right) + \left(\frac{1}{2}-\beta\right) \mathcal{L}_N (u_s^N) + \beta \mathcal{L}_N (u_{s+1}^N) \notag\\
&\quad-\left(\frac{1}{2}-\beta\right)\mathcal{L}_N (u_{s-1}^N) - \beta \mathcal{L}_N(u_s^N)\notag\\
&=\beta \mathcal{L}_N (u_{s+1}^N) +(1-2\beta)\mathcal{L}_N(u_s^N) + \beta \mathcal{L}_N (u_{s-1}^N).\notag
\end{align}

We observe that, when $\beta=0$, this method coincides with the St\"ormer-Verlet method, which is explicit.

To find the displacement at each time step, we solve the non-linear system
\[F(u_{s+1}^N) = u_{s+1}^N - u_s^N - \Delta t v_s^N - (1-2\beta)\frac{(\Delta t)^2}{2} \mathcal{L}_N (u_s^N) - \beta (\Delta t)^2 \mathcal{L}_N (u_{s+1}^N) = 0,\]
by using, for example, the Newton iterative method.

\subsection{Convergence of the fully discrete approximation}
\label{sec:time-conv}

In this section, we investigate the convergence of the sequence $\{u_s^N\}_{s=0}^{S_T}$ to the exact solution of the problem~\eqref{eq:nonlinperid}. For the sake of simplicity, we limit our attention to the linear problem, namely, we consider the case $w(\eta) = \eta$. Throughout this section the notation  $u(t)$, for each $ t$, denotes a function depending on the space variable, namely $u(t)(\cdot)=u(\cdot,t)$ with $u(\cdot ,t)$ in a suitable space, analogously,   $u_{s}^N$ for each $s$ denotes  a function depending on the space variable. The following Lemmas are preliminary to the convergence result.

\begin{lemma}
\label{lm:sigma1}
Let $u$ be the solution of the problem~\eqref{eq:nonlinperid} with initial condition $u_0$, $v_0\in H^2_p(V)$. Suppose $u\in\mathcal{C}^3\left(H^2_p(V), [0,T]\right)$ and let $\{u_s^N\}_{s=0}^{S_T}$ be the sequence generated by the method~\eqref{eq:newmarkb}, then
\begin{equation}
\label{eq:normsigma1}
\norm{u_1^N-P_N u(\Delta t)}_{H^2_p(V)} \le M\left(\Delta x\right)^2 \left(\norm{u_0}_{H^2_p(V)} + \norm{v_0}_{H^2_p(V)}\right).
\end{equation}
\end{lemma}

\begin{proof}
Thanks to the regularity assumptions on $u$ with respect to the time variable, we can apply the Cauchy's mean value theorem to $u(t)$. For all $x\in V$, we have
\begin{align}
\label{eq:Cauchy}
u(\Delta t) &= u_0 + \Delta t v_0+\frac{(\Delta t)^2}{2} \left( (1-2\beta)u_{tt}(0) + 2\beta u_{tt}(\Delta t)\right) + R\\
&= u_0 + \Delta t v_0+\frac{(\Delta t)^2}{2} \left((1-2\beta)\mathcal{L}\left(u_0\right) + 2\beta \mathcal{L}(u(\Delta t))\right) + R,\notag
\end{align}
where $R = \mathcal{O}((\Delta t)^3)$. By~\eqref{eq:newmarkb} for $s=0$, we find
\begin{equation}
\label{eq:us1}
u_1^N = u_0^N + \Delta t v_0^N + \frac{(\Delta t)^2}{2}\left((1-2\beta) \mathcal{L}_N (u_0^N) + 2\beta \mathcal{L}_N(u_1^N)\right).
\end{equation}

We define $\sigma_{s} = u_{s}^N - P_N u(t_s)$. Then, the duality relation~\eqref{eq:duality},~\eqref{eq:Cauchy} and~\eqref{eq:us1} imply
\begin{align}
\label{eq:s0+s1}
\sigma_0 + \sigma_1 & = \sigma_0 + u_1^N - P_N u(\Delta t) \\
&=2\sigma_0 + \left(\Delta t\right)\left(v_0^N - P_N v_0\right) - R.\notag
\end{align}

We make the inner product of~\eqref{eq:s0+s1} with the term $\sigma_0 + \sigma_1$. 
Thus, using the Cauchy's inequality and Lemma~\ref{lm:rhostima}, we find
\begin{align}
\label{eq:innproduct1}
\norm{\sigma_0+\sigma_1}_{H^2_p(V)}^2 &=\left(\sigma_0+\sigma_1,\ \sigma_0+\sigma_1\right)\\
&=\left(2\sigma_0 + \left(\Delta t\right)\left(v_0^N - P_N v_0\right) - R, \ \sigma_0 + \sigma_1\right)\notag\\
&\le M \norm{\sigma_0+\sigma_1}_{H^2_p(V)} \left(\norm{\sigma_0}_{H_p^2(V)} + (\Delta t)\norm{v_0^N - P_N v_0}_{H^2_p(V)} + \left(\Delta t\right)^3\right).\notag\\
&\le M \left(\Delta x\right)^2 \norm{\sigma_0+\sigma_1}_{H^2_p(V)} \left(\norm{u_0}_{H^2_p(V)} + \norm{v_0}_{H^2_p(V)}\right),\notag
\end{align}
for some $M>0$.


Therefore, we conclude
\begin{align}
\label{eq:sigma11}
\norm{\sigma_1}_{H^2_p(V)} &= \norm{u_1^N - P_N u(\Delta t)}_{H^2_p(V)} \le \norm{\sigma_0}_{H^2_p(V)}+ \norm{\sigma_1}_{H^2_p(V)}\\ &\le M (\Delta x)^2 \left(\norm{u_0}_{H^2_p(V)}+ \norm{v_0}_{H_p^2(V)}\right).\notag
\end{align}

\end{proof}

\begin{lemma}
\label{lm:normsigmas}
Let $u$ be the solution of the problem~\eqref{eq:nonlinperid} with initial condition $u_0$, $v_0\in H^2_p(V)$. Suppose $u\in\mathcal{C}^3\left(H^2_p(V), [0,T]\right)$ and let $\{u_s^N\}_{s=0}^{S_T}$ be the sequence generated by the method~\eqref{eq:newmarkb}. If $1/4\le\beta \le 1/2$, then
\begin{equation}
\label{eq:normsigmas}
\norm{u_s^N-P_N u(t_s)}_{H^2_p(V)} \le M\left(\left(\Delta x\right)^2 +\left(\Delta t\right)^2\right),
\end{equation}
for $s=0,\cdots,S_T$ and $M$ is a positive constant independent on $\Delta x$ and $\Delta t$.
\end{lemma}

\begin{proof}
We observe that the following relations hold
\begin{gather*}
\frac{u(t_{s+1}) - 2 u(t_s) + u(t_{s-1})}{(\Delta t)^2} = u_{tt}(t_s) + \tilde{R},\\
u(t_{s+1}) + 2u(t_s)+u(t_{s-1}) = 4 u(t_s) + \tilde{R},
\end{gather*}
where $\tilde{R}=\mathcal{O}((\Delta t)^2)$ is the rest of the Taylor expansion.

The previous relations, the duality equation~\eqref{eq:duality}, the trinomial recurrence formulation of the method~\eqref{eq:trinomial} and the definition of the problem\eqref{eq:nonlinperid} imply
\begin{align}
\label{eq:calc}
\frac{\sigma_{s+1} - 2 \sigma_s + \sigma_{s-1}}{(\Delta t)^2} \ -& \beta \mathcal{L}_N (\sigma_{s+1}) - (1-2\beta) \mathcal{L}_N(\sigma_s) - \beta \mathcal{L}_N(\sigma_{s-1})\\
=&\frac{u_{s+1}^N - 2 u_s^N + u_{s-1}^N}{(\Delta t)^2} - P_N\left(\frac{u(t_{s+1}) - 2 u(t_s) + u(t_{s-1})}{(\Delta t)^2}\right)\notag\\
&-\beta \mathcal{L}_N(u_{s+1}^N) - (1-2\beta)\mathcal{L}_N(u_s^N) - \beta\mathcal{L}_N(u_{s-1}^N)\notag\\
&+\beta\mathcal{L}_N(P_N u_{s+1}^N) + (1-2\beta)\mathcal{L}_N(P_N u_s^N) + \beta \mathcal{L}_N(P_N u_{s-1}^N)\notag\\
=& -P_N u_{tt}(t_s) + \beta\mathcal{L} (u(t_{s+1})) + (1-2\beta) \mathcal{L}(u(t_s)) + \beta \mathcal{L}(u(t_{s-1}))\notag\\
=& -P_N u_{tt}(t_s) + \beta u_{tt}(t_{s+1}) + 2\beta u_{tt}(t_s) + \beta u_{tt}(t_{s-1})\notag\\
& +u_{tt}(t_s) - 4\beta u_{tt}(t_s)\notag\\
=& - P_N u_{tt}(t_s) + u_{tt}(t_s) + \tilde{R}.\notag
\end{align}

%

Let us define $\vfi_{s+\frac{1}{2}} = \frac{\sigma_{s+1}-\sigma_s}{\Delta t},$ so,
\begin{equation}
\label{eq:phi}
\frac{\sigma_{s+1} - 2 \sigma_s + \sigma_{s-1}}{(\Delta t)^2} = \ \frac{\vfi_{s+\frac{1}{2}} - \vfi_{s-\frac{1}{2}}}{\Delta t},\qquad
\vfi_{s+\frac{1}{2}} + \vfi_{s-\frac{1}{2}} = 2 \ \frac{\sigma_{s+\frac{1}{2}} - \sigma_{s-\frac{1}{2}}}{\Delta t}.
\end{equation}

Now, we consider the inner product of~\eqref{eq:calc} with $(\vfi_{s+\frac{1}{2}} + \vfi_{s-\frac{1}{2}})$. Using the relations~\eqref{eq:phi}, for the first term on the left-side of~\eqref{eq:calc} we find 
\begin{equation}
\label{eq:1left}
\left(\frac{\vfi_{s+\frac{1}{2}} - \vfi_{s-\frac{1}{2}}}{\Delta t},\ \vfi_{s+\frac{1}{2}} + \vfi_{s-\frac{1}{2}}\right) = \frac{1}{\Delta t} \left(\norm{\vfi_{s+\frac{1}{2}}}_{H^2_p(V)}^2 - \norm{\vfi_{s-\frac{1}{2}}}_{H_p^2(V)}^2\right).
\end{equation}
For the second term on the left-side of~\eqref{eq:calc}, using the spectral properties of the discrete peridynamic operator $\mathcal{L}_N$, we get
\begin{align}
\label{eq:2left}
-2&\left(\beta\mathcal{L}_N (\sigma_{s+1}) + (1-2\beta)\mathcal{L}_N (\sigma_{s}) + \beta\mathcal{L}_N(\sigma_{s-1}), \ \vfi_{s+\frac{1}{2}} + \vfi_{s-\frac{1}{2}}\right) \\
 &\qquad = -\frac{4\beta}{\Delta t} \left(\mathcal{L}_N (\sigma_{s+\frac{1}{2}}), \ \sigma_{s+\frac{1}{2}}\right)+\frac{4\beta}{\Delta t}\left(\mathcal{L}_N (\sigma_{s-\frac{1}{2}}),\ \sigma_{s-\frac{1}{2}}\right).\notag
\end{align}

Let us focus on the right-side of~\eqref{eq:calc}. Lemma~\ref{lm:rhostima} and the Cauchy-Schwartz inequality ensure
\begin{align}
\label{eq:1right}
\left((u_{tt}(t_s) - P_N u_{tt}(t_s)), \ \vfi_{s+\frac{1}{2}} + \vfi_{s-\frac{1}{2}}\right) +& \left(\tilde{R},\ \vfi_{s+\frac{1}{2}} + \vfi_{s-\frac{1}{2}}\right) \notag\\
\le M&  \norm{u_{tt}(t_s) - P_N u_{tt}(t_s)}_{H_p^2(V)}\norm{\vfi_{s+\frac{1}{2}} + \vfi_{s-\frac{1}{2}}}_{H^2_p(V)} \\ 
&+ M(\Delta t)^2 \norm{\vfi_{s+\frac{1}{2}} + \vfi_{s-\frac{1}{2}}}_{H^2_p(V)}\notag\\
\le M& \left(\norm{u_{tt}(t_s) - P_N u_{tt}(t_s)}_{H_p^2(V)}^2 + (\Delta t)^4 \right) \notag\\ 
&+ M\left(\norm{\vfi_{s+\frac{1}{2}}}_{H^2_p(V)}^2 + \norm{\vfi_{s-\frac{1}{2}}}_{H^2_p(V)}^2\right)\notag\\
\le M&  \left((\Delta x)^2 + (\Delta t)^4 + \norm{\vfi_{s+\frac{1}{2}}}_{H^2_p(V)}^2 + \norm{\vfi_{s-\frac{1}{2}}}_{H^2_p(V)}^2\right).\notag
\end{align}

Merging~\eqref{eq:1left}, \eqref{eq:2left} and~\eqref{eq:1right} into~\eqref{eq:calc}, we have
\begin{align}
\label{eq:merge}
\frac{1}{\Delta t} &\left(\norm{\vfi_{s+\frac{1}{2}}}_{H^2_p(V)}^2-\norm{\vfi_{s-\frac{1}{2}}}_{H^2_p(V)}^2 -4\beta\left(\mathcal{L}_N (\sigma_{s+\frac{1}{2}}),\ \sigma_{s+\frac{1}{2}}\right) + 4\beta\left(\mathcal{L}_N (\sigma_{s-\frac{1}{2}}),\ \sigma_{s-\frac{1}{2}}\right)\right)\\
&\qquad\le M \left((\Delta x)^2 + (\Delta t)^4 + \norm{\vfi_{s+\frac{1}{2}}}_{H^2_p(V)}^2 +\norm{\vfi_{s-\frac{1}{2}}}_{H^2_p(V)}^2\right).\notag
\end{align}

Adding to the both side of inequality~\eqref{eq:merge}, the term
\begin{align*}
\frac{1}{\Delta t} &\left(\norm{\sigma_{s+1}}_{H_p^2(V)}^2 -\norm{\sigma_{s-1}}_{H^2_p(V)}^2\right) = \left(\sigma_{s+1} + \sigma_{s-1}, \vfi_{s+\frac{1}{2}} + \vfi_{s-\frac{1}{2}}\right)\\
&\le M\left(\norm{\sigma_{s+1}}_{H_p^2(V)}^2 + 2\norm{\sigma_s}_{H_p^2(V)}^2 +\norm{\sigma_{s-1}}_{H_p^2(V)}^2 + \norm{\vfi_{s+\frac{1}{2}}}_{H_p^2(V)}^2 +\norm{\vfi_{s-\frac{1}{2}}}_{H_p^2(V)}^2\right),
\end{align*}
we obtain
\begin{align}
\label{eq:sum}
\frac{1}{\Delta t} &\left(\norm{\vfi_{s+\frac{1}{2}}}_{H^2_p(V)}^2-\norm{\vfi_{s-\frac{1}{2}}}_{H^2_p(V)}^2 -4\beta\left(\mathcal{L}_N (\sigma_{s+\frac{1}{2}}),\ \sigma_{s+\frac{1}{2}}\right) + 4\beta\left(\mathcal{L}_N( \sigma_{s-\frac{1}{2}}),\ \sigma_{s-\frac{1}{2}}\right)\right)\\
&\qquad +\frac{1}{\Delta t}\left(\norm{\sigma_{s+1}}_{H_p^2(V)}^2 -\norm{\sigma_{s-1}}_{H^2_p(V)}^2\right)\notag\\
& \le M \left(\norm{\sigma_{s+1}}_{H_p^2(V)}^2 + 2\norm{\sigma_s}_{H_p^2(V)}^2 +\norm{\sigma_{s-1}}_{H_p^2(V)}^2 + \norm{\vfi_{s+\frac{1}{2}}}_{H_p^2(V)}^2 +\norm{\vfi_{s-\frac{1}{2}}}_{H_p^2(V)}^2\right) \notag\\
&\qquad+ M\left((\Delta x)^2 + (\Delta t)^2\right)^2.\notag
\end{align}
We set
\begin{equation}
\label{eq:defgamma}
\Gamma_s = \norm{\vfi_{s+\frac{1}{2}}}_{H^2_p(V)}^2 + \norm{\sigma_{s+1}}_{H^2_p(V)}^2 + \norm{\sigma_s}_{H^2_p(V)}^2 -4\beta \left(\mathcal{L}_N(\sigma_{s+\frac{1}{2}}),\ \sigma_{s+\frac{1}{2}}\right),
\end{equation}
then, we can compact~\eqref{eq:sum} in the following way
\begin{equation*}
\frac{\Gamma_s - \Gamma_{s-1}}{\Delta t} \le M \left((\Delta x)^2 + (\Delta t)^2\right)^2 + M\left(\Gamma_s + \Gamma_{s-1}\right).
\end{equation*}
The Gronwall inequality implies
\begin{equation}
\label{eq:Gronwall}
\Gamma_s \le \left(\Gamma_0 + \sum_{k=1}^s \Delta t\left((\Delta x)^2+(\Delta t)^2\right)^2\right) e^{M t_s \Delta t}.
\end{equation}

Additionally, the same argument as in~\eqref{eq:innproduct1} and~\eqref{eq:innproduct2} allows us to obtain
\begin{equation}
\label{eq:gamma0}
\Gamma_0 \le M\left((\Delta x)^2 + (\Delta t)^2\right)^2.
\end{equation}
Hence, since $t_s\le T$, using the definition of $\Gamma_s$ in~\eqref{eq:defgamma} and collecting~\eqref{eq:Gronwall} and~\eqref{eq:gamma0}, we find
\begin{equation}
\label{eq:chain}
\norm{\sigma_s}_{H_p^2(V)}^2 \le \Gamma_s \le M \left((\Delta x)^2 + (\Delta t)^2\right)^2.
\end{equation}
Therefore, the estimate~\eqref{eq:chain} ensures that
\begin{equation}
\label{eq:sigmas}
\norm{u_s^N-P_Nu(t_s)}_{H^2_p(V)} \le M ((\Delta x)^2 + (\Delta t)^2). 
\end{equation}
\end{proof}

The following convergence result holds.
\begin{theorem}
\label{th:tconv}
Let $u$ be the solution of the problem~\eqref{eq:nonlinperid} with initial condition $u_0$, $v_0\in H^2_p(V)$. Suppose $u\in\mathcal{C}^3\left(H^2_p(V), [0,T]\right)$ and let $\{u_s^N\}_{s=0}^{S_T}$ be the sequence generated by the method~\eqref{eq:newmarkb}. If $1/4\le\beta \le 1/2$, then
\begin{equation}
\label{eq:tdifference}
\norm{u(t_s) - u_s^N}_{H^2_p(V)} \le M \left(\Delta x\right)^2 \left(\norm{u_0}_{H^2_p(V)}+ \norm{u_t}_{L^1(H^2_p(V),0,t_s)} + \norm{v_0}_{H^2_p(V)}\right) +M \left(\Delta t\right)^2,
\end{equation}
where $s=0,\cdots, S_T$ and $M>0$ is a constant depending on the regularity of $u$ and independent on $\Delta x$ and $\Delta t$.
\end{theorem}

\begin{proof}
Using the triangular inequality, we have
\begin{equation}
\label{eq:triangular}
\norm{u(t_s) - u_s^N}_{H^2_p(V)} \le \norm{u(t_s)-P_N u(t_s)}_{H^2_p(V)} + \norm{u_s^N- P_N u(t_s)}_{H^2_p(V)}.
\end{equation}

Lemma~\ref{lm:rhostima} implies
\begin{align}
\label{eq:rhos}
\norm{u(t_s) - P_N u(t_s)}_{H^2_p(V)} &\le M\, (\Delta x)^2 \norm{u(t_s)}_{H^2_p(V)}\\
 &= M\,(\Delta x)^2\left(\norm{u_0}_{H^2_p(V)} + \int_0^{t_s}\norm{u_t(r)}_{H^2_p(V)}dr\right)\notag\\
&\le M\,(\Delta x)^2 \left(\norm{u_0}_{H^2_p(V)} + \norm{u_t}_{L^1(H^2_p(V),0,t_s)}\right).\notag
\end{align}

Now, we focus on the difference $\sigma_s=(u_s^N- P_N u(t_s))$. We start by considering the case $s=0$: Lemma~\ref{lm:rhostima} implies again
\begin{equation}
\label{eq:sigma0}
\norm{u_0^N- P_N u(t_0)}_{H^2_p(V)} = \norm{u_{n,0}- P_N u_0}_{H^2_p(V)} \le M\,(\Delta x)^2\norm{u_0}_{H^2_p(V)}.
\end{equation}

If $s=1$ and $t_1=\Delta t$, thanks to Lemma~\ref{lm:sigma1} we have
\begin{equation}
\label{eq:sigma1}
\norm{\sigma_1}_{H^2_p(V)} \le M (\Delta x)^2 \left(\norm{u_0}_{H^2_p(V)}+ \norm{v_0}_{H_p^2(V)}\right).
\end{equation}

We turn on the general case $s\ge1$. Lemma~\ref{lm:normsigmas} ensures that
\begin{equation}
\label{eq:sigmass}
\norm{u_s^N-P_Nu(t_s)}_{H^2_p(V)} = \norm{\sigma_s}_{H_p^2(V)} \le M ((\Delta x)^2 + (\Delta t)^2). 
\end{equation}

Therefore, using~\eqref{eq:rhos} and~\eqref{eq:sigmass} into~\eqref{eq:triangular}, we complete the proof.
\end{proof}

\subsection{Stability of the Newmark-$\beta$ method}
\label{sec:stability}

In this section, we prove the stability of the method by the energy method, showing that the norm of the numerical solution admits a sublinear behavior with respect to the time variable.

\begin{theorem}
\label{th:stab}

Let $\{u_s^N\}_{s=0}^{S_T}$ be the sequence generated by the method~\eqref{eq:newmarkb}. If $1/4 \le \beta\le 1/2$, then there exist two positive constants $M_0$ and $M_1$ such that
\begin{equation}
\label{eq:stab-estimate}
\norm{u^N_s}_{H_p^2(V)} \le M_1 + M_0 t_s,\qquad s=0,\cdots,S_T.
\end{equation}
\end{theorem}

\begin{proof}
Let us define the test function $\psi^N_{s+\frac{1}{2}} = \frac{u_{s+1}^N - u_s^N}{\Delta t}$, and     
\[u_{s+1/2}^N = \frac{u_{s+1}^N+ u_{s}^N}{2},\qquad \text{for $s=0,\dots, S_T -1$.}\]
so
\begin{equation}
\label{eq:psi}
\frac{\psi^N_{s+\frac{1}{2}} - \psi^N_{s-\frac{1}{2}}}{\Delta t} = \frac{u_{s+1}^N - 2 u_s^N + u_{s-1}^N}{(\Delta t)^2},\qquad \psi^N_{s+\frac{1}{2}} + \psi^N_{s-\frac{1}{2}} = \frac{u_{s+1}^N - u_{s-1}^N}{\Delta t}.
\end{equation}

If we take the inner product between the equation~\eqref{eq:trinomial} and $\left(\psi^N_{s+\frac{1}{2}} + \psi^N_{s-\frac{1}{2}}\right)$, we obtain the following energy equation
\begin{align}
\label{eq:energy}
&\left(\frac{\psi^N_{s+\frac{1}{2}} - \psi^N_{s-\frac{1}{2}}}{\Delta t},\ \psi^N_{s+\frac{1}{2}} + \psi^N_{s+\frac{1}{2}}\right) - \beta\left(\mathcal{L}_N(u_{s+1}^N),\ \psi^N_{s+\frac{1}{2}} + \psi^N_{s+\frac{1}{2}}\right)\\
&\qquad - (1-2\beta)\left(\mathcal{L}_N(u_s^N),\ \psi^N_{s+\frac{1}{2}} + \psi^N_{s+\frac{1}{2}}\right) - \beta \left(\mathcal{L}_N(u_{s-1}^N),\ \psi^N_{s+\frac{1}{2}} + \psi^N_{s+\frac{1}{2}}\right) =0.\notag
\end{align}

Thanks to the spectral properties of the discrete peridynamic operator $\mathcal{L}_N$, we get
\begin{equation}
\label{eq:lnpsi}
\left(\mathcal{L}_N(u^N_{s+\frac{1}{2}}) + \mathcal{L}_N(u^N_{s-\frac{1}{2}}),\psi^N_{s+\frac{1}{2}}+\psi^N_{s-\frac{1}{2}}\right) = \left(\mathcal{L}_N(u^N_{s+\frac{1}{2}}),u^N_{s+\frac{1}{2}}\right) - \left(\mathcal{L}_N(u^N_{s-\frac{1}{2}}),u^N_{s-\frac{1}{2}}\right).
\end{equation}

Then, using~\eqref{eq:energy} and~\eqref{eq:lnpsi}, we have
\begin{align}
\label{eq:ener}
\left(\frac{\psi^N_{s+\frac{1}{2}} - \psi^N_{s-\frac{1}{2}}}{\Delta t},\ \psi^N_{s+\frac{1}{2}} + \psi^N_{s-\frac{1}{2}}\right) &-2\beta\left(\mathcal{L}_N(u^N_{s+\frac{1}{2}}),u^N_{s+\frac{1}{2}}\right)\\
& - 2\beta\left(\mathcal{L}_N(u^N_{s-\frac{1}{2}}),u^N_{s-\frac{1}{2}}\right) \le 0.\notag
\end{align}

Multiplying both sides of~\eqref{eq:ener} by $\left(\Delta t\right)^3$ and adding $u_s^N - u_s^N =0$ in the inner product of the first term in the left hand side of~\eqref{eq:ener} and finally using the definition of $\psi^N_{s+\frac{1}{2}}$, we find
\begin{align}
\label{eq:rearrange}
\left(u_{s+1}^N-u_s^N,u_{s+1}^N - u_s^N\right) &- \beta\left(\mathcal{L}_N(u^N_{s+\frac{1}{2}}), u^N_{s+\frac{1}{2}}\right)\\
&\le \left(u_{s}^N-u_{s-1}^N,u_{s}^N - u_{s-1}^N\right) - \beta\left(\mathcal{L}_N(u^N_{s-\frac{1}{2}}),u^N_{s-\frac{1}{2}}\right).\notag
\end{align}

Therefore, dividing both sides by $(\Delta t)^2$ and using a recurrence argument we obtain
\begin{align*}
&\left(\frac{u_{s+1}^N-u_s^N}{\Delta t},\frac{u_{s+1}^N - u_s^N}{\Delta t}\right) - \frac{\beta}{(\Delta t)^2}\left(\mathcal{L}_N(u^N_{s+\frac{1}{2}}), u^N_{s+\frac{1}{2}}\right)\\
&\quad\le  \left(\frac{u_1^N-u_0^N}{\Delta t},\frac{u_1^N - u_0^N}{\Delta t}\right) - \frac{\beta}{(\Delta t)^2}\left(\mathcal{L}_N\left(u^N_{\frac{1}{2}}\right),u^N_{\frac{1}{2}}\right).\\
\end{align*}
Hence,
\begin{equation}
\label{eq:concl}
\norm{\frac{u_{s+1}^N-u_s^N}{\Delta t}}_{H_p^2(V)}^2 -\frac{\beta}{(\Delta t)^2} \left(\mathcal{L}_N(u^N_{s+\frac{1}{2}}), u^N_{s+\frac{1}{2}}\right)\le M_0,
\end{equation}
where $M_0 = \norm{\frac{u_1^N-u_0^N}{\Delta t}}_{H_p^2(V)}^2- \frac{\beta}{(\Delta t)^2}\left(\mathcal{L}_N\left(u^N_{\frac{1}{2}}\right),u^N_{\frac{1}{2}}\right) \ge 0$,
as Lemma~\ref{lm:spectral} ensures that $-\left(\mathcal{L}_N(\omega),\omega\right) \ge 0$ for all $\omega$.

We notice that $\norm{\frac{u_{s+1}^N-u_s^N}{\Delta t}}_{H_p^2(V)}^2$ and $-\frac{\beta}{(\Delta t)^2} \left(\mathcal{L}_N(u^N_{s+\frac{1}{2}}), u^N_{s+\frac{1}{2}}\right)$ are positives. As a consequence, thanks to~\eqref{eq:concl} we find
\[
\norm{u_{s+1}^N-u_s^N}_{H_p^2(V)}\le M_0 \Delta t,
\]
and
\begin{align*}
\norm{u^N(t_s+\Delta t)}_{H_p^2(V)} &\le \norm{u^N(t_s)}_{H_p^2(V)} + M_0\Delta t \\
&\le \norm{u^N(t_s - \Delta t)}_{H_p^2(V)} + 2M_0\Delta t\\
&\le\cdots\le \norm{u^N(0)}_{H_p^2(V)} + M_0 \left(t_s + \Delta t\right).
\end{align*}
Thus, if we choose $M_1=\norm{u^N(0)}_{H_p^2(V)}$, we get the thesis.
\end{proof}


\section{Numerical tests}
\label{sec:test}

In this section we validate the proposed scheme and collect several simulations in order to investigate the properties of the solutions of the nonlinear peridynamic model~\eqref{eq:nonlinperid}.
\subsection[Validation of spectral semi-discretization scheme]{Validation of spectral semi-discretization scheme}
\label{sec:validation}

To validate the results of the peridynamic scheme, we implement the following 2D benchmark problem. We consider a thin lamina in the spatial domain $[0,1]\times[0,1]$ and we discretize it with a bi-dimensional mesh using the same space step $\Delta x = 0.01$ on both directions. We assume that the lamina is subjected to the uniform initial displacement $u_0(x_1,x_2) = -0.5 x_1-0.5x_2$, and we fix $\delta =0.2$ as horizon. 

We take the micromodulus function $C(x_1,x_2) =\exp{(-x_1^2-x_2^2)}$, and we choose $w(\eta) = \eta^r$, with $r=3$. Moreover, we assume that the body is not subjected to external forces, namely $b\equiv 0$ and the constant density of the body is $\rho(x_1,x_2) = 1$. For the implementation of the Newmark-$\beta$ method, we take $\beta=1/4$.

\begin{figure}
\centering
\begin{subfigure}[b]{.48\textwidth}
\includegraphics[width=\textwidth]{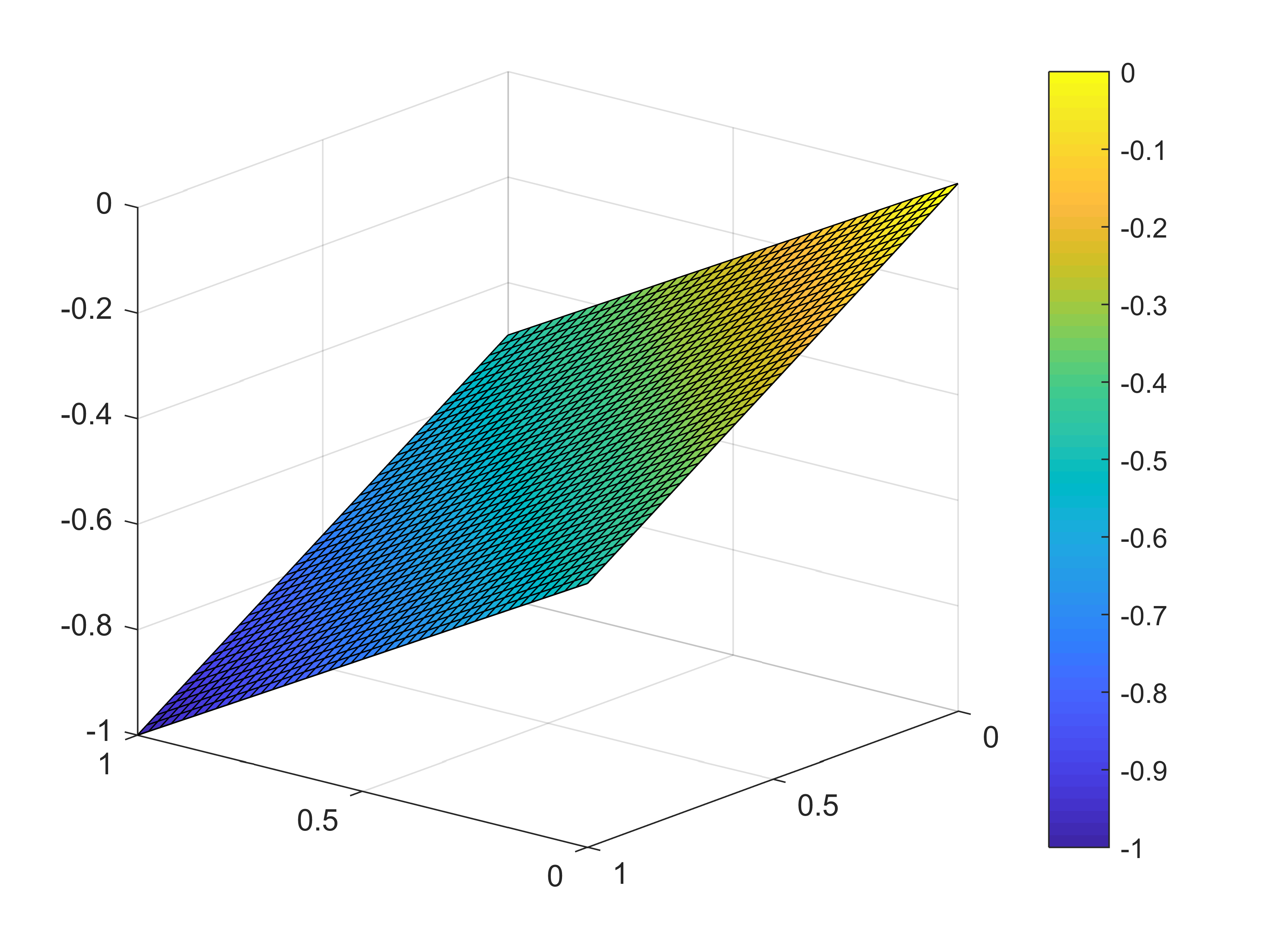}
\caption*{
$t=0$.}
\end{subfigure}
\begin{subfigure}[b]{.48\textwidth}
\includegraphics[width=\textwidth]{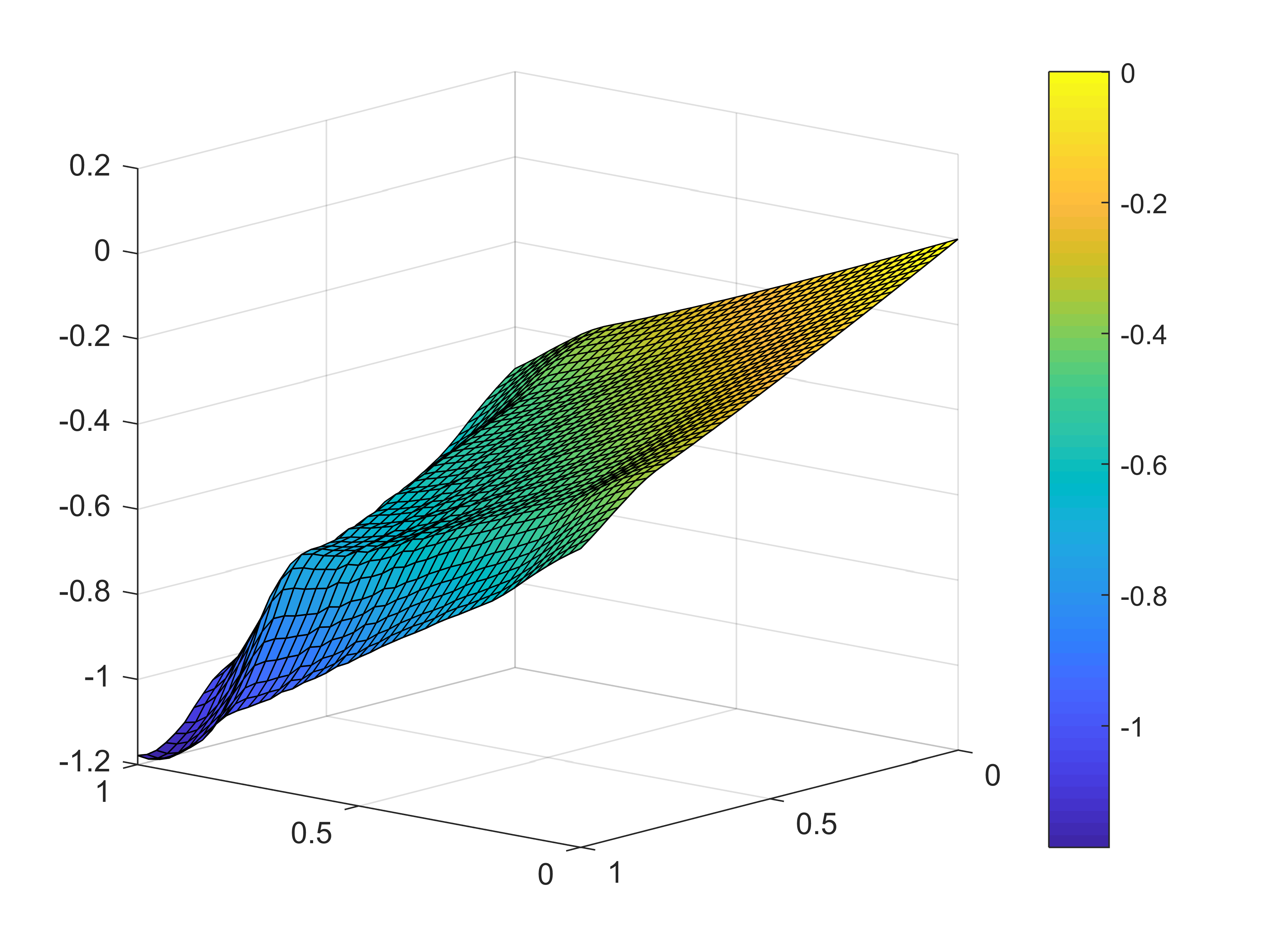}
\caption*{
$t=5$.}
\end{subfigure}
\\
\begin{subfigure}[b]{.48\textwidth}
\includegraphics[width=\textwidth]{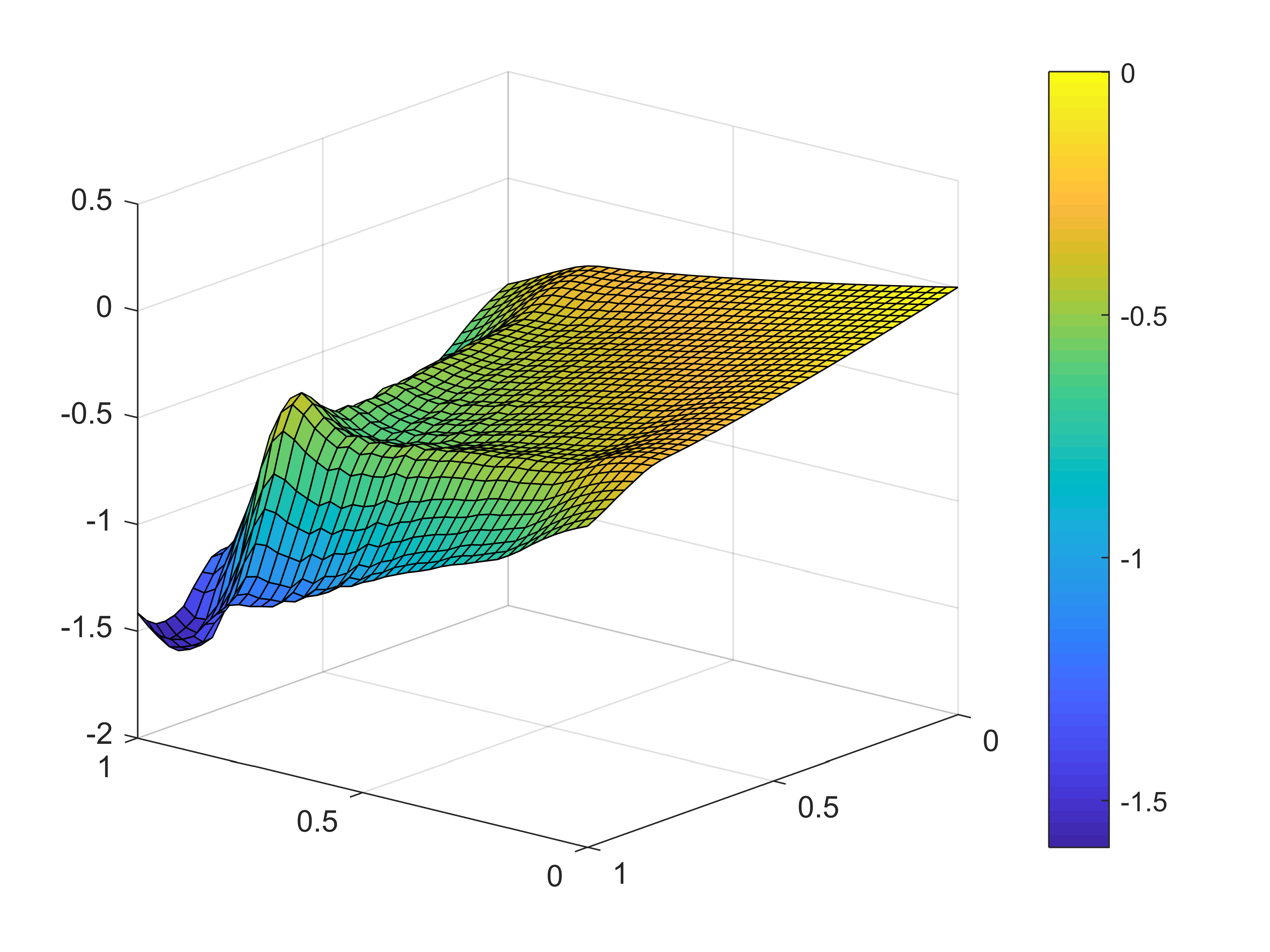}
\caption*{
$t=7.5$.}
\end{subfigure}
\begin{subfigure}[b]{.48\textwidth}
\includegraphics[width=\textwidth]{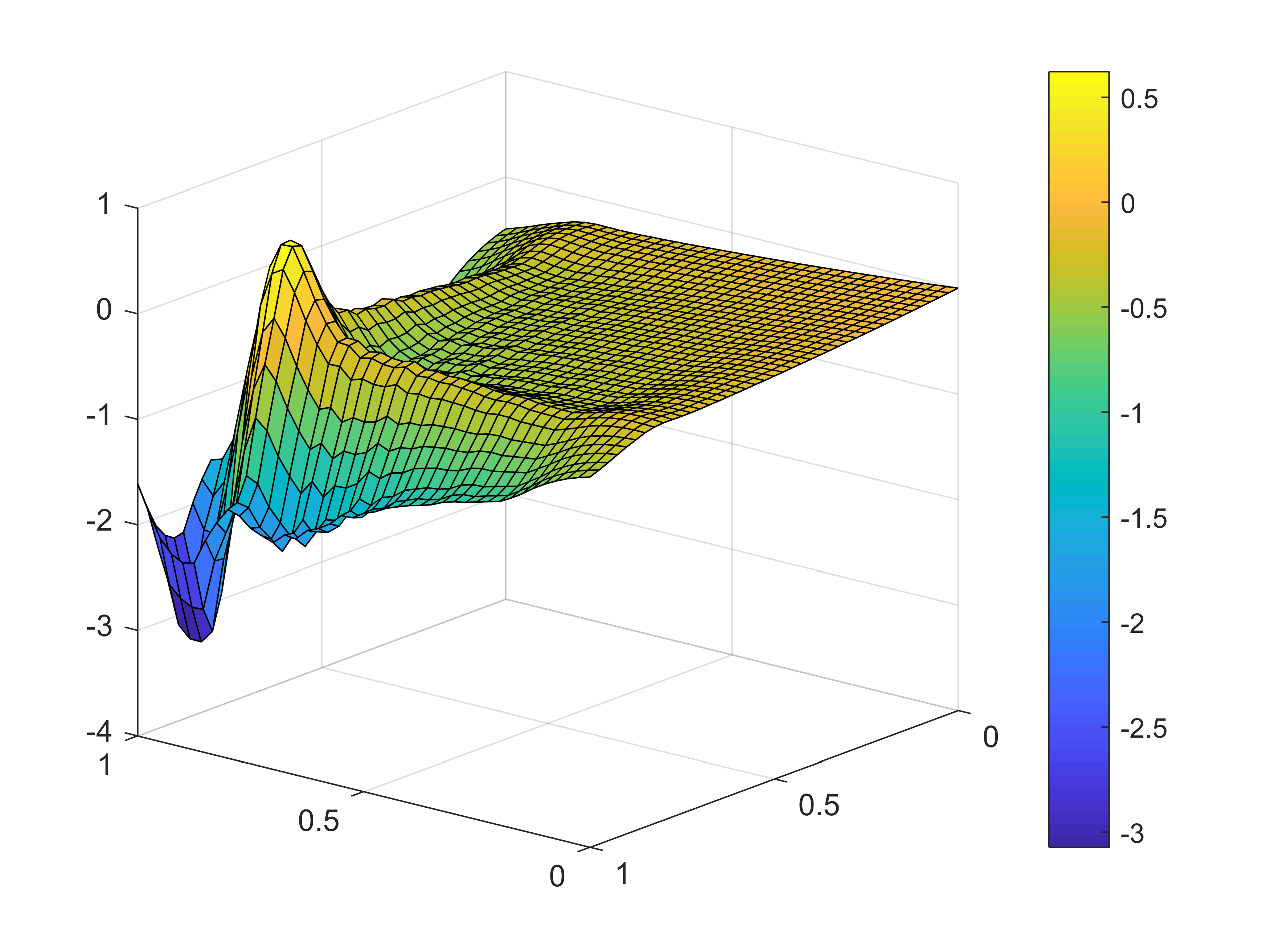}
\caption*{$t=10$.}
\end{subfigure}
\caption{The initial condition and the solution at times $t=2.5$,$t=5$, $t=7.5$ and $t=10$. The parameters for the simulation are $\delta=0.2$, $\Delta x=10^{-2}$, and $\Delta t=10^{-4}$.}
\label{fig:evolution}
\end{figure}

In Figure~\ref{fig:evolution}, we plot the initial condition and the behavior of the solution in the spatial domain as time evolves. The convergence of the fully-discrete scheme is evaluated by computing the relative error in the discrete $L^2(V)$ norm at time $t_s$:
\[
E^{t_s}_{L^2} = \frac{\sum_{n}\left|u^N_{n,s}- u^*(x_n,t_s)\right|^2}{\sum_{n}\left|u^N_{n,s}\right|^2},
\]
where $u^*$ denotes the reference solution for the problem. 

We notice that finding an exact solution of a non-linear problem is a not trivial issue. Therefore, in this work we determine $u^*$ using our scheme with a finer mesh.

In Table~\ref{tab:error} we choose a very small time step $\Delta t= 10^{-4}$ and we depict the error $E^{t_s}_{L^2}$ between the exact solution and the numerical one for different value of the space step $\Delta x$ at time $t_s=5$. We can also observe that the rate of convergence of the scheme seems in accordance with the theoretical results about the accuracy of the method.

\begin{table}%
\centering%
\renewcommand\arraystretch{1.3}
\begin{tabular}{ccc
}
\toprule
$\Delta x$& $E_{L^2}^{t_s}$ &convergence rate
\\
\midrule
$0.2$&$5.1194 \times 10^{-1}$&$-$
\\
$0.1$&$6.8616\times 10^{-2}$&$2.8994$
\\
$0.05$&$1.2494\times 10^{-2}$&$2.6783$
\\
$0.025$&$2.1966\times10^{-3}$&$2.6051$
\\
$0.0125$&$54138\times10^{-4}$&$2.4735$
\\
\bottomrule
\end{tabular}
\renewcommand\arraystretch{1}
\caption{The relative $L^2$-error at time $t_s=5$ as function of the space step, for $\Delta t = 10^{-4}$.}
\label{tab:error}
\end{table}

Additionally, we perform a convergence analysis also with respect to the time step. Using the same setting and data as before, we fix $\Delta x = 10^{-4}$ and we compute the error $E^{t_s}_{L^2}$ between the exact solution and the numerical one for different values of the time step $\Delta t$ at time $t_s=5$. Table~\ref{tab:error-time} shows the values of the relative error and that the convergence rate seems to be in accordance with the theoretical results.

\begin{table}%
\centering%
\renewcommand\arraystretch{1.3}
\begin{tabular}{ccc
}
\toprule
$\Delta t$& $E_{L^2}^{t_s}$ &convergence rate
\\
\midrule
$0.1$&$1.1271 \times 10^{-6}$&$-$
\\
$0.05$&$2.0212\times 10^{-7}$&$2.4793$
\\
$0.025$&$7.3230\times 10^{-8}$&$1.9720$
\\
$0.01$&$4.9246\times10^{-9}$&$2.2943$
\\
$0.005$&$6.5037\times10^{-10}$&$2.4587$
\\
\bottomrule
\end{tabular}
\renewcommand\arraystretch{1}
\caption{The relative $L^2$-error at time $t_s=5$ as function of the time step, for $\Delta x = 10^{-4}$.}
\label{tab:error-time}
\end{table}

Moreover, in order to overcome the limitation of periodic boundary conditions due to the spectral spatial discretization, we use a volume penalization technique. We recall that the penalization procedure extend the computational domain $V$ to a fictitious one $\Omega$ by a factor $\mu>0$, in order that
\[
\Omega = V \cup \Gamma,
\]
where $\Gamma$ denotes the constrained domain, see Figure~\ref{fig:penal} . It imposes the periodic boundary conditions to the extended domain and then penalizes the solution on $\Gamma$ by means of a penalization term, which depends on a factor $\eps>0$ called penalization factor. It results that the penalization term converges to zero as the penalization factor $\eps$ goes to zero. For a complete description of the technique in the one-dimensional case, we refer the reader to~\cite{LP}.

We validate the spectral method with volume penalization by making a comparison between the exact solution and the numerical one. We work in the same setting as before and we fix $\eps=0.2$ as penalization factor. Table~\ref{tab:error-penal} summarizes the error study.

\begin{figure}[!htbp]
\centering
\includegraphics[width=0.4\textwidth]{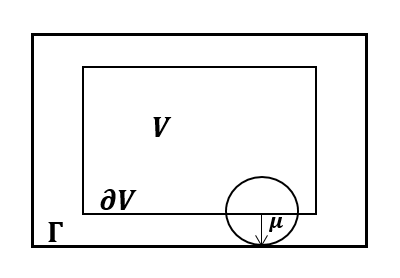}
\caption{With reference to Section~\ref{sec:validation}: the reference domain $V$, its boundary $\partial V$ and the constrained domain $\Gamma$. The fictitious domain is $\Omega=V\cup \Gamma$.}
\label{fig:penal}
\end{figure}

\begin{table}%
\centering%
\renewcommand\arraystretch{1.3}
\begin{tabular}{ccc
}
\toprule
$\Delta x$& $E_{L^2}^{t_s}$ &convergence rate
\\
\midrule
$0.2$&$7.8142\times 10^{-1}$&$-$
\\
$0.1$&$1.2049\times10^{-1}$&$2.6972$
\\
$0.05$&$2.5370\times 10^{-2}$&$2.4725$
\\
$0.025$&$6.1826\times 10^{-3}$&$2.3193$
\\
$0.01$&$8.2514\times 10^{-4}$&$2.2570$
\\
\bottomrule
\end{tabular}
\renewcommand\arraystretch{1}
\caption{With reference to Section~\ref{sec:validation}, the relative $L^2$-error corresponding to the spectral method with volume penalization at time $t_s=5$ as function of the space step, for $\Delta t = 10^{-4}$.}
\label{tab:error-penal}
\end{table}

\subsection{Comparison between Newmark-$\beta$ and St\"ormer-Verlet methods}
\label{sec:compare}

In this section we test the performance of the two methods with respect to the time step $\Delta t$. For a description of the St\"ormer-Verlet method, we refer the reader to~\cite{Galvanetto2016,LP,CFLMP}.

We take under consideration a thin lamina in the spatial domain $[0,1]\times[0,1]$ and we discretize it with a bi-dimensional mesh using the same space step $\Delta x = 0.01$ on both directions. We choose $u_0(x_1,x_2) = -0.5 x_1-0.5x_2$ as initial condition, and we fix $\delta =0.2$ as horizon. We choose the same parameters as in the previous test. 

In Table~\ref{tab:compare}, we compute the relative $L^2$ error at time $t_s=5$ between the exact solution and the numerical one obtained both with the Newmark-$\beta$ method and the St\"ormer Verlet method for decreasing time step values. We can observe that the Newmark-$\beta$ method allows us to have the same accuracy of the St\"ormer-Verlet scheme, but using a greater time step.

\begin{table}%
\centering%
\renewcommand\arraystretch{1.3}
\begin{tabular}{ccc
}
\toprule
\multirow{2}*{$\Delta t$}& \multicolumn{2}{c}{$E_{L^2}^{t_s}$} \\ \cmidrule(lr){2-3}
& Newmark-$\beta$ & St\"ormer-Verlet\\

\midrule
$0.1$&$4.6812 \times 10^{-6}$&$5.1511\times 10^{-4}$
\\
$0.05$&$2.9276\times 10^{-7}$&$3.4650\times 10^{-5}$
\\
$0.01$&$4.6629\times 10^{-9}$&$3.4634\times 10^{-5}$
\\
$0.005$&$2.8706\times10^{-10}$&$3.0868\times 10^{-6}$
\\
$0.001$&$2.6360\times10^{-12}$&$2.3476 \times 10^{-7}$
\\
\bottomrule
\end{tabular}
\renewcommand\arraystretch{1}
\caption{With reference to Section~\ref{sec:compare}, the relative $L^2$-error at time $t_s=5$ as function of the time step, for $\Delta x = 0.01$.}
\label{tab:compare}
\end{table}

\subsection{The case of a discontinuous initial condition}
\label{sec:jump}
In this section, we study the behavior of the solution when the initial condition is discontinuous. We consider $[0,1]\times[0,1]$ as domain of computation. We take the micromodulus function as in the sections above, $r=3$ and we fix the size of the horizon as $\delta = 0.2$. We choose a jump-type discontinuity $u_0(x_1,x_2) = \chi_{[1/2,1]\times [1/2,1]}(x_1,x_2)$, and $v_0(x_1,x_2)=0$ as initial condition. Figure~\ref{fig:3djump} shows the evolution of the solution at times $t=3.5$, $t=4$, $t=5.5$ and $t=6.5$. We can notice the formation of waves travelling with different phase speeds, with an increasing amplitude and a decreasing wavelength propagating from the discontinuous point. The parameters for the simulation are $\Delta x = 10^{-2}$, $\Delta t =10^{-4}$ and $\beta =1/4$.

\begin{figure}
\centering
\begin{subfigure}[b]{.48\textwidth}
\includegraphics[width=\textwidth]{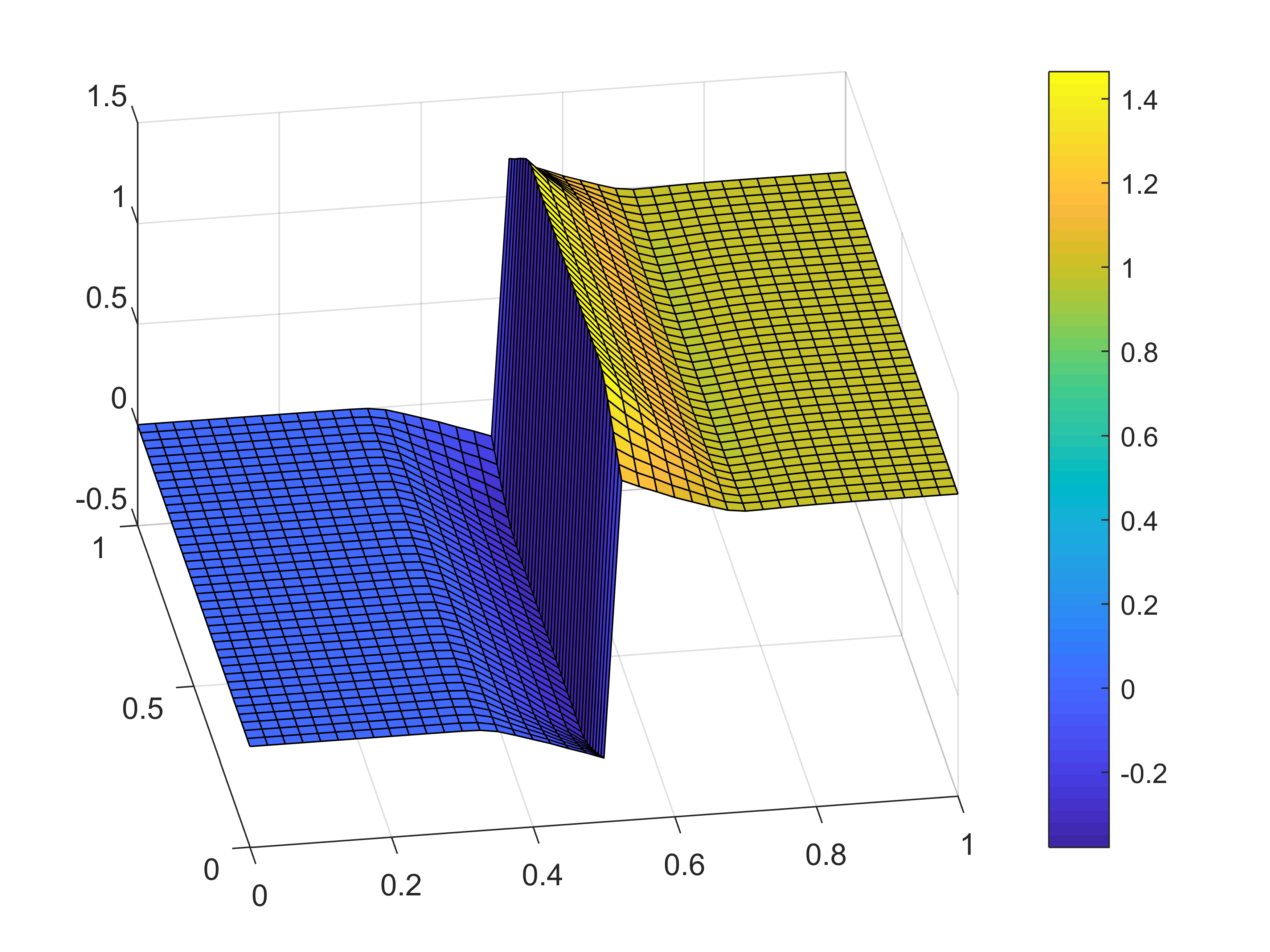}
\caption*{$u(x,3.5)$}
\end{subfigure}
\begin{subfigure}[b]{.48\textwidth}
\includegraphics[width=\textwidth]{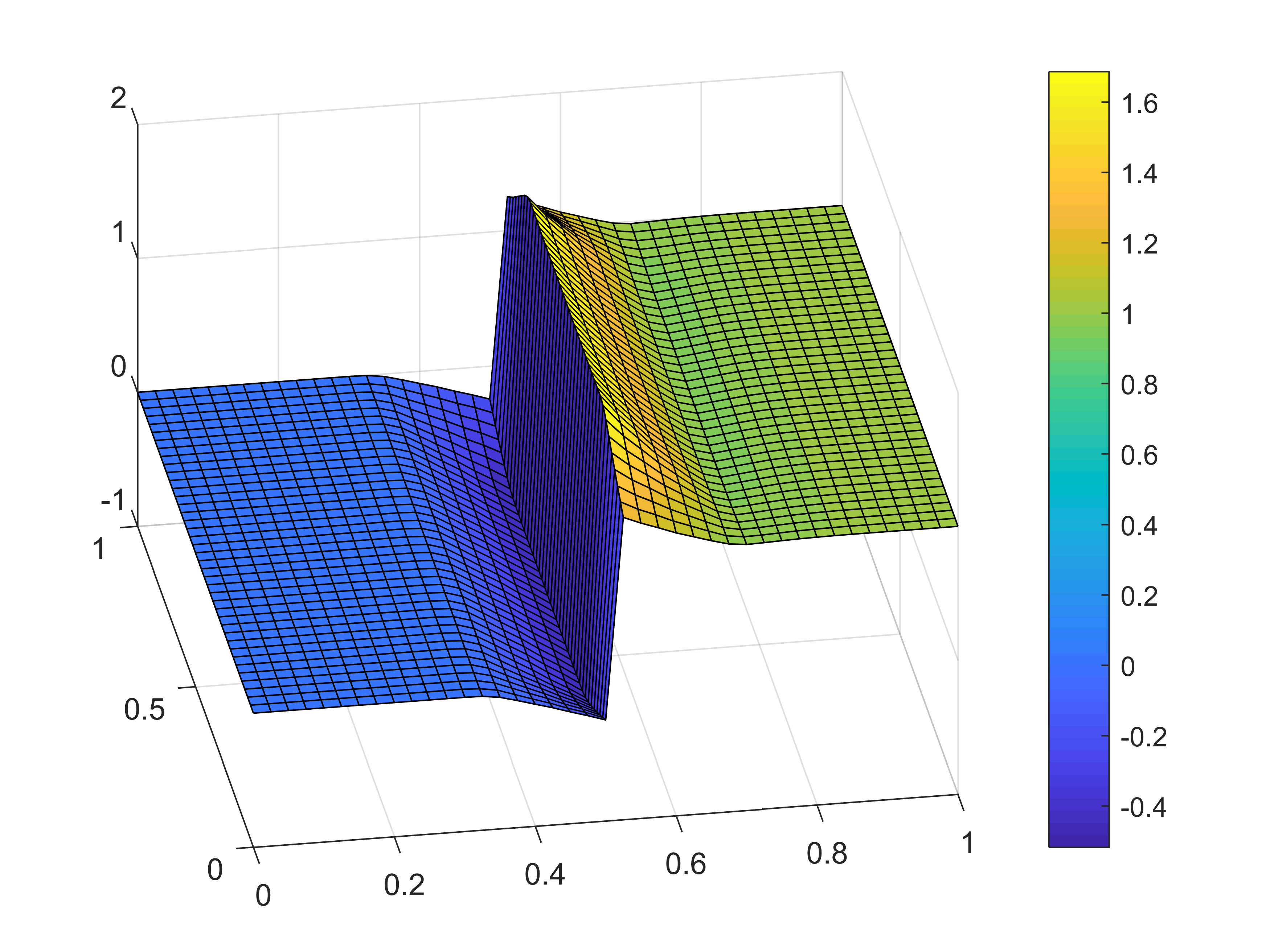}
\caption*{$u(x,4)$}
\end{subfigure}
\begin{subfigure}[b]{.48\textwidth}
\includegraphics[width=\textwidth]{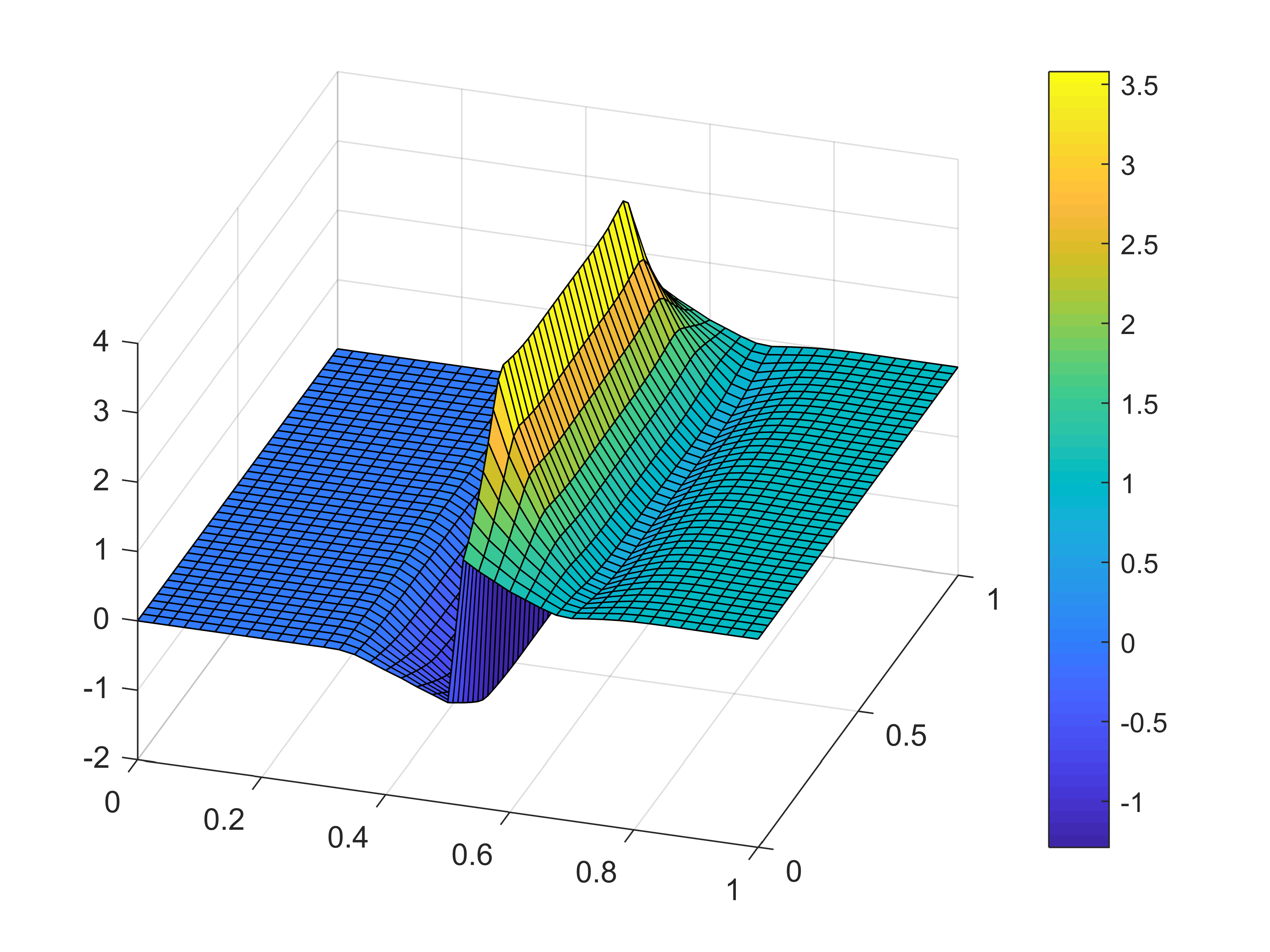}
\caption*{$u(x,5.5)$}
\end{subfigure}
\begin{subfigure}[b]{.48\textwidth}
\includegraphics[width=\textwidth]{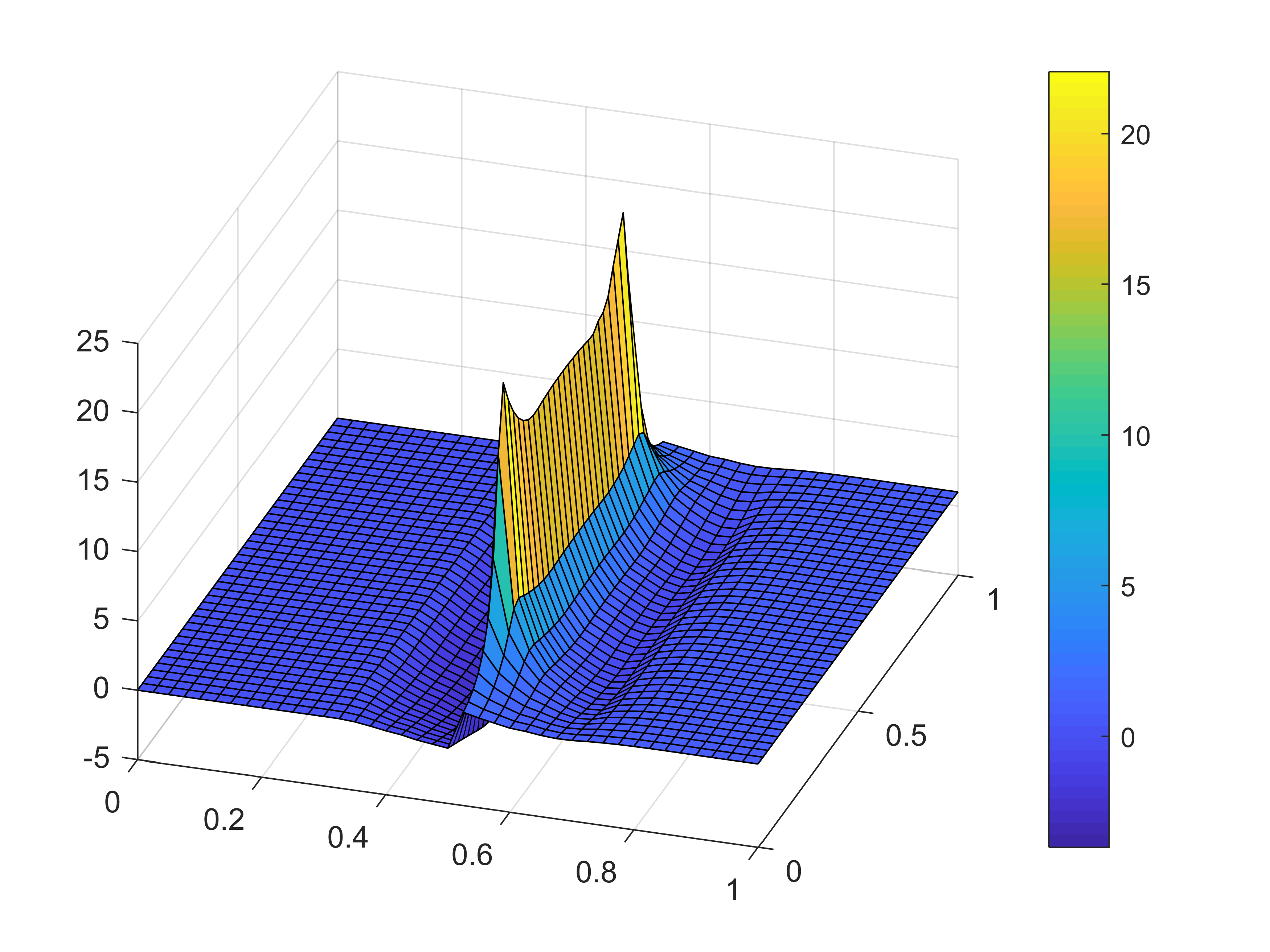}
\caption*{$u(x,6.5)$}
\end{subfigure}
\caption{With reference to Section~\ref{sec:jump}: the evolution of the solution for $\delta=0.2$ corresponding to the initial condition $u_0(x) = \chi_{[1/2,1]\times [1/2,1]}$, $v_0(x)=0$. The parameters for the computed solution are $\beta =1/4$, $\Delta x = 10^{-2}$ and $\Delta t =10^{-4}$.}
\label{fig:3djump}
\end{figure}

Moreover, we perform an error analysis also in this case. Table~\ref{tab:error-discontinuous} shows the lost an order of convergence due to the presence of a discontinuity in the initial condition.

\begin{table}%
\centering%
\renewcommand\arraystretch{1.3}
\begin{tabular}{ccc
}
\toprule
$\Delta x$& $E_{L^2}^{t_s}$ &convergence rate
\\
\midrule
$0.2$&$3.4418\times 10^{-1}$&$-$
\\
$0.1$&$1.2339\times10^{-1}$&$1.4799$
\\
$0.05$&$5.9369\times 10^{-2}$&$1.2677$
\\
$0.025$&$2.5997\times 10^{-2}$&$1.2236$
\\
$0.0125$&$8.2514\times 10^{-4}$&$2.2570$
\\
\bottomrule
\end{tabular}
\renewcommand\arraystretch{1}
\caption{With reference to Section~\ref{sec:jump}, the relative $L^2$-error at time $t_s=5$ as function of the space step, for $\Delta t = 10^{-4}$.}
\label{tab:error-discontinuous}
\end{table}

\section{Conclusions and future works}
\label{sec:conclusion}

In this paper a new bi-dimensional peridynamic discretization model has been proposed. It is based on a spectral Fourier discretization for the spatial domain and the implementation of the Newmark-$\beta$ method for the time marching. We have recalled a convergence result for the semi-discrete problem and we have proved the convergence of the fully-discrete linear problem. 

Our results shows that spectral techniques perform very well in the nonlinear case and the Newmark-$\beta$ method allows us to have a good accuracy without using a too small time step.

In future, we would like to extend the analytical result on the convergence of the fully discrete scheme to the nonlinear case. Moreover, we aim to couple our approach to techniques based on finite element methods or mimetic finite difference methods (see for example~\cite{Beirao_Lopez_Vacca_2017,Lopez_Vacca_2016}), following the same strategy proposed in~\cite{Galvanetto2016}.

\section*{Acknowledgements}
This paper has been partially supported by GNCS of Istituto Nazionale di Alta Matematica and by PRIN 2017 ``Discontinuous dynamical systems: theory, numerics and applications'' coordinated by Nicola Guglielmi.


\bibliographystyle{plain}
\bibliography{biblioPeri2}

\end{document}